\documentclass[11pt]{amsart}

\usepackage{amsmath,amssymb,amsthm}
\usepackage{color, tikz}
\usepackage{tabularx}
\usepackage{float}
\usepackage[linesnumbered,lined,boxed,commentsnumbered]{algorithm2e}
\usepackage{textcomp}

\makeatletter
\newcommand{\addresseshere}{%
  \enddoc@text\let\enddoc@text\relax
}
\makeatother

\usepackage[margin=1in]{geometry}
\theoremstyle{plain}
\newtheorem{theorem}{Theorem}[section]
\newtheorem{prop}[theorem]{Proposition}

\newtheorem{lemma}[theorem]{Lemma}
\newtheorem{corollary}[theorem]{Corollary}
\newtheorem{problem}[theorem]{Problem}
\theoremstyle{definition}
\newtheorem{remark}[theorem]{Remark}
\newtheorem{definition}[theorem]{Definition}

\newcommand{\vid}{\mathrm{vid}}
\newcommand{\haj}{\Delta_H}
\newcommand{\ore}{\mathrm{Ore}}
\newcommand{\urq}{\mathrm{Urq}}
\newcommand{\N}{\mathcal{N}}
\newcommand{\dd}{\mathrm{d}}
\newcommand{\link}{\mathrm{link}}
\newcommand{\cstar}{\mathrm{star}}
\newcommand{\da}{\mathrm{NP}}

\newcommand{\mbf}[1]{\mathbf{#1}}

\newcommand{\todo}[1]{\par \noindent
  \framebox{\begin{minipage}[c]{0.95 \textwidth} TO DO:
      #1 \end{minipage}}\par}
\newcommand{\bnote}[1]{\par \noindent
  \framebox{\begin{minipage}[c]{0.95 \textwidth}\color{blue} BEN'S NOTE:
      #1 \color{black}\end{minipage}}\par}
\definecolor{green}{RGB}{34, 139, 34}
\newcommand{\jnote}[1]{\par \noindent
  \framebox{\begin{minipage}[c]{0.95 \textwidth}\color{green} JULIE'S NOTE:
      #1 \color{black}\end{minipage}}\par}
\newcommand\commentout[1]{}

\begin{document}

% title header etc
\title{Haj\'os-Type Constructions and Neighborhood Complexes}

\author{Benjamin Braun}
\address{Department of Mathematics\\
         University of Kentucky\\
         Lexington, KY 40506--0027}
\email{benjamin.braun@uky.edu}

\author{Julianne Vega}
\address{Department of Mathematics\\
         University of Kentucky\\
         Lexington, KY 40506--0027}
\email{julianne.vega@uky.edu}

\date{18 December 2018}

%\thanks{
%}

\begin{abstract}
  Any graph $G$ with chromatic number $k$ can be constructed by iteratively performing certain graph operations on a sequence of graphs starting with $K_k$, resulting in a variety of Haj\'os-type constructions for $G$.
  Finding such constructions for a given graph or family of graphs is a challenging task.
  We show that the basic steps in these Haj\'os-type constructions frequently result in the presence of an $S^1$-wedge summand in the neighborhood complex of the resulting graph.
  Our results imply that for a graph $G$ with a highly-connected neighborhood complex, the end behavior of the construction sequence is quite restricted, and we investigate these restrictions in detail.
  We also introduce two graph construction algorithms based on different Haj\'os-type constructions and conduct computational experiments using these.
\end{abstract}

\maketitle

%%%%%%%%%%%%%%%%%%%%%%%%%%%%%%%%%%%%%%%%%%%%%%%%%%%%%%%%%
\section{Introduction}\label{sec:intro}

\subsection{Motivation}
Proper graph colorings are of great interest in combinatorics and the study of chromatic numbers for graphs is a frequent focus.
While coarse bounds on the chromatic number can be found easily, determining the chromatic number is NP-complete.
In 1978, Lov\'asz advanced the study of chromatic numbers through the construction of the neighborhood complex $\mathcal{N}(G)$ of a graph $G$.
Intuitively, $\mathcal{N}(G)$ is capturing the relations of the vertices with their neighbors and the topological connectivity of $\N(G)$ is measuring the complexity of continuous deformations of the neighborhoods in the graph.
Lov\'asz~\cite{LovaszChromaticNumberHomotopy} proved that the topological connectivity of $\N(G)$ gives a general lower bound for the chromatic number of $G$, and then showed that this provides a sharp lower bound for Kneser graphs.
Since this original result, there has been steady development regarding our understanding of neighborhood complexes of graphs and various topological lower bounds for chromatic numbers~\cite{KozlovBook}.

Another area of interest related to graph colorings is the characterization of $k$-chromatic graphs, i.e. graphs with chromatic number $k$.
In 1961, Haj\'os  \cite{hajos} characterized graphs with chromatic number at least $k$ through the concept of a $k$-constructible subgraph, where constructibility is defined recursively using the operations of vertex identification and Haj\'os merging.
Strengthening this result, Urquhart ~\cite{urquhart} showed that every $k$-chromatic graph is $t$-constructible for $t \leq k$.
In the proof, Urquhart describes a Haj\'os-type construction which involves an ``Ore merge,'' an operation that involves the Haj\'os merge followed by a restricted series of vertex identifications.

Jensen and Royle ~\cite{jensenroyle} considered Haj\'os' result in terms of $k$-critical graphs.
They proved that there exists $k$-critical graphs that are not Haj\'os constructible through a sequence of $k$-critical graphs.
The proof of this result involved exploring the operations necessary for the final steps of the Haj\'os construction sequence, showing that the key to constructing $k$-critical graphs is to end in a vertex identification.

Building on these narratives, we are led to ask in this article how the connectivity of the neighborhood complex of a given graph is affected by the Haj\'os merge and vertex identification operations. In addition, we ask what end behavior of Haj\'os construction sequences is necessary to obtain graphs with highly (topologically) connected neighborhood complexes.

\subsection{
Our Contributions
}
We explore the topological effects of Haj\'os-type constructions from both theoretical and experimental perspectives.
In Section~\ref{sec:background}, we discuss background for Haj\'os-type constructions, neighborhood complexes of graphs, and discrete Morse theory.
We explore the topological effects of Haj\'os merges and vertex identifications in Section~\ref{sec:operations}, providing insight into restrictions on Haj\'os-type constructions for graphs with highly connected neighborhood complexes.
Our main result is Corollary~\ref{cor:vid5top}, which states that such constructions must end with an identification of vertices at distance four or less from each other in the graph.
Motivated by this, in Section~\ref{sec:shortdistance} we investigate the impact of ``short-distance'' identifications on the first Betti number of $\N(G)$.
In Section~\ref{sec:dhgo}, we briefly investigate the topological effects of DHGO compositions, a generalization of Haj\'os merges.
Finally, in Section~\ref{sec:experiments}, we introduce two graph construction algorithms based on different Haj\'os-type constructions, discuss the outcomes of computational experiments using these, and conclude with open problems.

%%%%%%%%%%%%%%%%%%%%%%%%%%%%%%%%%%%%%%%%%%%%%%%%%%%%%%%%%
\section{Background}\label{sec:background}

%%%%%%%%%%%%%%%%%%%%%%%%%%%%%%%%%%%%%%%%%%%%%%%%%%%%%%%%%%%%
\subsection{Haj\'os-type Constructions}

In this section we review Haj\'os-type constructions.
Recall that a graph is \emph{$k$-chromatic} if it has chromatic number $k$.

\begin{definition}\label{def:hajosconst}
A graph is called \emph{Haj\'os $k$-constructible} if it is a complete graph $K_k$ or if it can be constructed from $K_k$ by successive applications of the following two operations:
\begin{itemize}
\item (Haj\'os Merge) If $G_1$ and $G_2$ are already-obtained disjoint graphs, then to the disjoint union $G_1 \uplus G_2$ remove an edge $(x_1,y_1)$ from $G_1$ and an edge $(x_2,y_2)$ from $G_2$, identify $x_1$ with $x_2$, and add the edge $(y_1,y_2)$. 
We abuse notation and denote the resulting graph $G_1\haj G_2$.
\item (Vertex Identification) Identify two nonadjacent vertices in an already-obtained graph $H$, where we ignore the presence of multiple edges.
If $L$ is a list of pairs of nonadjacent vertices in $G$, then $\vid(G,L)$ is the graph obtained by identifying all those pairs of vertices.
\end{itemize}
\end{definition}

Any construction of a graph using this process will be called a \emph{Haj\'os construction}.
An example of a Haj\'os merge is given in Figure~\ref{fig:hajosk3}. 
It is straightforward to show that $\chi(G_1\haj G_2)\geq \min(\chi(G_1),\chi(G_2))$, and that $\chi(\vid(G,L))\geq \chi(G)$.
Thus, if $G$ is Haj\'os $k$-constructible, then $\chi(G)\geq k$.
Haj\'os proved further that if $\chi(G)\geq k$, then $G$ contains a Haj\'os $k$-constructible subgraph.

\begin{figure}
\begin{center}
\begin{tikzpicture}[scale = 0.5, every node/.style={}]

\draw[line width=1pt] (0,0) -- (0,2);
\draw[red, line width = 1.5pt] (0,2) -- (2,0);
\draw[line width=1pt] (2,0) -- (0,0);

\draw[line width=1pt] (5,0) -- (7,0);
\draw[line width=1pt] (7,0) -- (7,2);
\draw[red, line width=1.5pt] (7,2) -- (5,0);

\fill[] (0,0) circle (6pt); 
\fill[] (0,2) circle (6pt); 
\fill[] (2,0) circle (6pt); 

\fill[] (5,0) circle (6pt); 
\fill[] (7,0) circle (6pt); 
\fill[] (7,2) circle (6pt); 

\node (a)  at (0, -0.75) {$z_1$};
\node (b) at (2, -0.75) {$x_1$};
\node (c) at (0, 2.5) {$y_1$};

\node (d) at (5, -0.75) {$x_2$};
\node (e) at (7, -0.75) {$z_2$};
\node (f) at (7, 2.5) {$y_2$};

\draw[->] (7.5, 1) to (9, 1); 

\draw[line width=1pt] (12,0) -- (13,2);
\draw[blue, line width = 1.5pt] (13,2) -- (15,2);
\draw[line width=1pt] (12,0) -- (14,0);

\draw[line width=1pt] (14,0) -- (16,0);
\draw[line width=1pt] (16,0) -- (15,2);

\fill[] (14,0) circle (6pt); 
\fill[] (12,0) circle (6pt); 
\fill[] (16,0) circle (6pt); 

\fill[] (15,2) circle (6pt); 
\fill[] (13,2) circle (6pt);

\node (g) at (12, -0.75) {$z_1$};
\node (h) at (14, -0.75) {$x_1x_2$};
\node (i) at (16, -0.75) {$z_2$};

\node (j) at (13, 2.5) {$y_1$};
\node (k) at (15, 2.5) {$y_2$};

\end{tikzpicture}
\end{center}
\caption{An example of a Haj\'os merge of $K_3$ with $K_3$.}
\label{fig:hajosk3}
\end{figure}
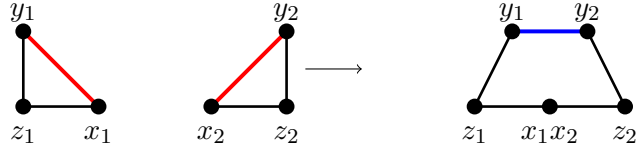

\begin{theorem}[H\'ajos~\cite{hajos}]
For every $3\leq t\leq k$, every $k$-chromatic graph contains a $t$-constructible subgraph.
\end{theorem}\label{thm:hajos}

Urquhart later strengthened Haj\'os' result.

\begin{theorem}[Urquhart~\cite{urquhart}]\label{thm:urquhart}
For every $3\leq t\leq k$, every $k$-chromatic graph is $t$-constructible.
\end{theorem}

This theorem has been the subject of continued investigation in recent years~\cite{iwama,jensengrassman,jensenroyle,johnsonhajos,kralhajos,liuzhang,urquhart}.
Urquhart proved Theorem~\ref{thm:urquhart} using a variant of H\'ajos' construction that was introduced by Ore~\cite{orefourcolor} as follows.

\begin{definition}\label{def:oreconst}
A graph is \emph{Ore $k$-constructible} if it is a complete graph $K_k$ or if it can be constructed from $K_k$ by successive applications of the following operation:
\begin{itemize}
\item (Ore Merge) Suppose $G_1$ and $G_2$ are already-obtained disjoint graphs with respective edges $(x_1,y_1)$ and $(x_2,y_2)$. 
Let $\mu$ be a bijection from a subset of $V(G_1)$ to $V(G_2)$ so that $x_1\notin \mathrm{domain}(\mu)$ and $x_2\notin \mathrm{range}(\mu)$, and $\mu(y_1)\neq y_2$.
Form the Haj\'os merge on $G_1$ and $G_2$ using the two edges, then identify the vertex pairs $[x,\mu(x)]$.
We abuse notation and denote the resulting graph $\ore(G_1,G_2)$.
\end{itemize}
Any construction of a graph using this process will be called an \emph{Ore construction}.
\end{definition}

Note that an Ore merge arises from a single Haj\'os merge followed by a sequence of restricted vertex identifications.
However, Urquhart~\cite{urquhart} proved that the families of Haj\'os constructible and Ore constructible graphs are equivalent.

\begin{theorem}[Urquhart~\cite{urquhart}]\label{thm:hajosore}
For a graph $G$ and $k \geq 2$, the following conditions are equivalent: 
\begin{itemize}
\item $G$ is Haj\'os-$k$-constructible; 
\item $G$ is Ore-$k$-constructible. 
\end{itemize}
\end{theorem} 

In the proof of Theorem~\ref{thm:hajosore}, Urquhart proved that for $k\geq 3$ every $k$-chromatic graph $G$ can be obtained by the following process.

\begin{definition}\label{def:urquhartconst}
Suppose that a graph $G$ is obtained by applying Ore merges to a sequence of graphs $G_1,G_2,\ldots,G_\ell$ each containing $K_k$.
We will call a construction of $G$ using this approach an \emph{Urquhart construction}, and denote (again abusing notation) $G=\urq(G_1,\ldots,G_\ell)$.
\end{definition}

A $k$-chromatic graph is $k$-critical if every proper subgraph is $j$-chromatic for some $j <k$. Theorem~ \ref{thm:hajos} implies that $k$-critical graphs are $k$-constructible. Jensen and Royle~\cite{jensenroyle}  proved the existence of $k$-critical graphs that do not have a Haj\'os sequence consisting of exclusively $k$-critical graphs.
\begin{theorem}
For every $k \geq 4$ there exists a $k$-critical graph that allows no Haj\'os $k$-construction where all intermediary graphs are $k$-critical. 
\end{theorem}

The proof involves finding graphs that satisfy the three specifications in the proposition below for $k\geq 4, k \neq 8$.  
\begin{prop}
If $G$
\begin{enumerate}
\item is $3$-connected,
\item has chromatic number at least $k$, and 
\item for every $v \in V(G)$ and every pair $u_1, u_2 \in N_G(v)$ there exists a $(k-1)$-coloring $\varphi$ of $G -v$ such that $\varphi(w) \neq \varphi(u_i)$ for all $w \in N_G(v) \smallsetminus \{u_i\}, i =1,2$,
\end{enumerate}
then $G$ is a $k$-critical graph such that the last step of any possible Haj\'os $k$-construction of $G$ consists of a vertex identification on a graph that is not critical. 
\end{prop}

%%%%%%%%%%%%%%%%%%%%%%%%%%%%%%%%%%%%%%%%%%%%%%%%%%%%%%%%%%%%%%%%%%%%
\subsection{Neighborhood Complexes}

Recall that a \textit{simplicial complex} $X$ on $n$ vertices is a collection of subsets of $[n]:=\{1,2,\ldots,n\}$ that is closed under containment and contains all singletons.  
We call an element of $X$ a \textit{face} of $X$.  
We will make no distinction between an abstract simplicial complex $X$ and an arbitrary geometric realization $|X|$ of $X$ as a topological space.   
A topological space $X$ is called $k$\textit{-connected} if for every $0\leq \ell \leq k$, every continuous map from the boundary of $B^{\ell+1}$, the unit ball in $(\ell+1)$-dimensional Euclidean space, into $X$ can be extended to a continuous map from all of $B^{\ell+1}$ to $X$.  
Equivalently, the higher homotopy groups $\pi_\ell(X)$ vanish for all dimensions $\ell\leq k$.

For any graph $G$, Lov\'asz~\cite{LovaszChromaticNumberHomotopy} defined the \textit{neighborhood complex of $G$}, denoted $\N(G)$, to be the simplicial complex with vertex set $V(G)$ and facets given by $N_G(v)$ for all $v\in V(G)$, where $N_G(v)$ denotes the neighbors of $v$ in $G$ (not including $v$).
Lov\'asz introduced $\N(G)$ in order to provide a sharp lower bound for the chromatic number of the Kneser graphs, which he did using the following theorem.
\begin{theorem}[Lov\'{a}sz~\cite{LovaszChromaticNumberHomotopy}]
If $\N(G)$ is $k$-connected, then $\chi(G)\geq k + 3$.
\end{theorem}
There are several famous families of graphs, e.g. Kneser and stable Kneser graphs, for which these topological lower bounds (or equivalent techniques) yield the only known proofs of their chromatic numbers.
Note that $\N(G)$ is $0$-connected if and only if it is path-connected, and being $0$-connected implies having chromatic number greater than $2$.
For connected bipartite graphs, having a disconnected $\N(G)$ characterizes this family.
The following proposition justifies our assumptions throughout this work that when $G$ is connected with $\chi(G)\geq 3$, $\N(G)$ is path-connected.

\begin{prop} \label{prop:connected}
The complex $\N(G)$ is path-connected if and only if $G$ is connected and not bipartite. 
\end{prop}

\begin{proof}
If $G$ is not connected, then it is immediate that $\N(G)$ is not connected.
Let $G$ be a connected bipartite graph with bipartition $V(G) = A \uplus B$, where $\uplus$ denotes disjoint union. 
For all $a \in A, N(a) \subseteq B$ and for all $b \in B, N(b) \subseteq A$. 
Therefore, the neighborhood complex induced by $B$ is disjoint from the neighborhood complex induced by $A$ and $\N(G) = \N(A) \uplus \N(B)$, hence $\N(G)$ is not connected.

Suppose now that $G$ is connected and not bipartite, so there exists an odd cycle $C = c_0, c_1,..., c_n$ in $G$.
We prove that between any two vertices $x,y \in V(G)$ there is a walk of even length, from which it follows that $x$ and $y$ are connected by a path in $\N(G)$.
Since $G$ is connected, there exists a walk $W_x$ from $x$ to some $c_x\in C$, and a walk $W_y$ from $y$ to some $c_y\in C$.
Since $C$ has odd length, there exists a walk $C_{odd}$ in $C$ of odd length from $c_x$ to $c_y$, and there also exists a walk $C_{even}$ in $C$ of even length from $c_x$ to $c_y$.
For any given even/odd parities of $W_x$ and $W_y$, one can connect $x$ and $y$ by a path of even length that starts with $W_x$, continues through either $C_{odd}$ or $C_{even}$, and concludes with $W_y$.
Thus, $x$ and $y$ are path-connected in $\N(G)$.
\end{proof}

In general, it is possible for there to be arbitrarily large gaps between $\chi(G)$ and $k+3$ where $\N(G)$ is $k$-connected, for example Matou{\v{s}}ek and Ziegler [Remark (H1) in~\cite{MatousekZiegler}] make a remark that implies the following.
\begin{prop}
If $G$ does not contain a $4$-cycle, then $\N(G)$ is at most $0$-connected.
\end{prop}
Thus, for example, the neighborhood complex for graphs with girth greater than $4$ are not $1$-connected.

%%%%%%%%%%%%%%%%%%%%%%%%%%%%%%%%%%%%%%%%%%%%%%%%%

\subsection{Discrete Morse Theory}

Discrete Morse theory was first developed by R. Forman in \cite{FormanMorseTheory} and has since become a powerful tool for topological combinatorialists.
The main idea of the theory is to pair faces within a simplicial complex in such a way that we obtain a sequence of collapses yielding a homotopy equivalent cell complex.

\begin{definition}\label{def:partialmatching}

A \textit{partial matching} in a poset $P$ is a partial matching in the underlying graph of the Hasse diagram
of $P$, i.e., it is a subset $M\subseteq P \times P$ such that
\begin{itemize}
\item
$(a,b)\in M$ implies $b \succ a;$ i.e. $a<b$ and no $c$ satisifies $a<c<b$.
\item
each $a\in P$ belongs to at most one element in $M$.
\end{itemize}
When $(a,b) \in M$, we write $a=d(b)$ and $b=u(a)$.
\item
A partial matching on $P$ is called \emph{acyclic} if there does not exist a cycle
\[
a_1 \prec u(a_1) \succ a_2 \prec u(a_2) \succ \cdots \prec u(a_m) \succ a_1
\]
with $m\ge 2$ and all $a_i\in P$ being distinct.

\end{definition}

Given an acyclic partial matching $M$ on a poset $P$, an element $c$ is \emph{critical} if it is unmatched.
If every element is matched by $M$, $M$ is called \emph{perfect}.
We are now able to state the main theorem of discrete Morse theory as given in~\cite[Theorem 11.13]{KozlovBook}

\begin{theorem}\label{thm:maindiscretemorse}
Let $\Delta$ be a polyhedral cell complex and let $M$ be an acyclic matching on the face poset of $\Delta$.
Let $c_i$ denote the number of critical $i$-dimensional cells of $\Delta$.
The space $\Delta$ is homotopy equivalent to a cell complex $\Delta_c$ with $c_i$ cells of dimension $i$ for each $i\ge 0$, plus a single $0$-dimensional cell in the case where the emptyset is paired in the matching.
\end{theorem}

It is often useful to create acyclic partial matchings on different sections of the face poset of a simplicial complex and then combine
them to form a larger acyclic partial matching on the entire poset.
This process is detailed in the following theorem known as the \textit{Cluster Lemma} in \cite{JonssonBook} and the \textit{Patchwork Theorem} in \cite{KozlovBook}.

\begin{theorem}\label{thm:patchwork}
Assume that $\varphi : P \rightarrow Q$ is an order-preserving map.
For any collection of acyclic matchings on the subposets $\varphi^{-1}(q)$ for $q\in Q$, the union of these matchings is itself an acyclic matching on $P$.
\end{theorem}

%%%%%%%%%%%%%%%%%%%%%%%%%%%%%%%%%%%%%%%%%%%%%%%%%%%%%%%
\section{Topological effects of Haj\'os-type operations}\label{sec:operations}

In this section, we investigate the effects  of Haj\'os merges and vertex identifications on $\N(G)$, with our main result being Corollary~\ref{cor:vid5top}.
We assume throughout this section that $G_1$ and $G_2$ are connected graphs with chromatic numbers at least $3$.
We begin by showing the neighborhood complex $\N(G_1 \haj G_2)$ is typically not $1$-connected.
Recall that a \emph{bridge} in a graph $G$ is an edge whose deletion increases the number of connected components of $G$.

\begin{lemma}
$G_1 \Delta_H G_2$ has a bridge if at least one of the edges used in the Haj\'os merge is a bridge. 
\end{lemma}

\begin{proof}
Suppose $G_1$ and $G_2$ are two connected graphs with edges $(x_1, y_1)$ and $(x_2,y_2)$ respectively. Suppose under the Haj\'os construction $y_1$ and $y_2$ get identified as $y_1y_2$, while $(x_1,y_1)$ and $(x_2,y_2)$ get deleted and $(x_1,x_2)$ is added to create $G_1 \haj G_2$. Let $G_i':= G_i \smallsetminus (x_i,y_i)$. Suppose that $(x_1,y_1)$ is a bridge in $G_1$ such that $G_1' := A \uplus B$, then the Hajos merge produces a graph with $(x_1, x_2)$ a bridge between $A$ and $\vid(G_2' \uplus B,[y_1,y_2])$. 
\end{proof}

\begin{lemma}\label{lem:bipartite}
Let $G$ be a connected non-bipartite graph such that $G' := G \smallsetminus (x,y)$ is connected bipartite. Then, $\mathcal{N}(G') = A \uplus B$, where $N_{G'}(x), N_{G'}(y) \subseteq A$ and $ x, y \in B$. 
\end{lemma} 

\begin{proof}
  Since the deletion of the edge $(x,y)$ results in a bipartite graph, $(x,y)$ must be part of all the odd cycles and no even cycles in $G$.
  Hence, by the connectivity of $G'$, there must be an even path $p$ from $x$ to $y$ and no such odd path.
  Therefore, there is no even path from $y$ to any neighbor $v \in N(x)$ or from $x$ to any neighbor $w \in N(y)$, which implies $y$ and $N_{G'}(x)$ are not in the same connected component.
  Similarly for $x$ and $N_{G'}(y)$.
  In addition, notice that $x$ and $N_{G'}(x)$ cannot be in the same component because they would need to be connected by an even path in $G'$ creating an odd cycle, a contradiction to $G'$ being bipartite.
  Similarly for $y$ and $N_{G'}(y)$.
  Since a connected bipartite graph gives rise to a neighborhood complex with two connected components, the result follows.
\end{proof}

\begin{theorem}\label{thm:hajostop}
For two connected graphs $G_1$ and $G_2$ with edges $(x_1, y_1)$ and $(x_2,y_2)$ respectively such that either
\begin{enumerate}
\item $\chi(G_1), \chi(G_2) \geq 3$ and neither $(x_1, y_1)$ nor $(x_2,y_2)$ is a bridge or
\item $\chi(G_1) \geq 3$, $\chi(G_2) \geq 4$, and $(x_2, y_2)$ is not a bridge,
\end{enumerate}
$\N(G_1 \haj G_2)$ is homotopy equivalent to a wedge of at least one copy of $S^1$ with another space.
\end{theorem}

\begin{proof}
Suppose under the Haj\'os construction $y_1$ and $y_2$ get identified as $y_1y_2$, while $(x_1,y_1)$ and $(x_2,y_2)$ get deleted and $(x_1,x_2)$ is added to create $G_1 \haj G_2$. Let $G_i':= G_i \smallsetminus (x_i,y_i)$.

%%%%%%%%%%%%%%%%%%%%%%%%%%

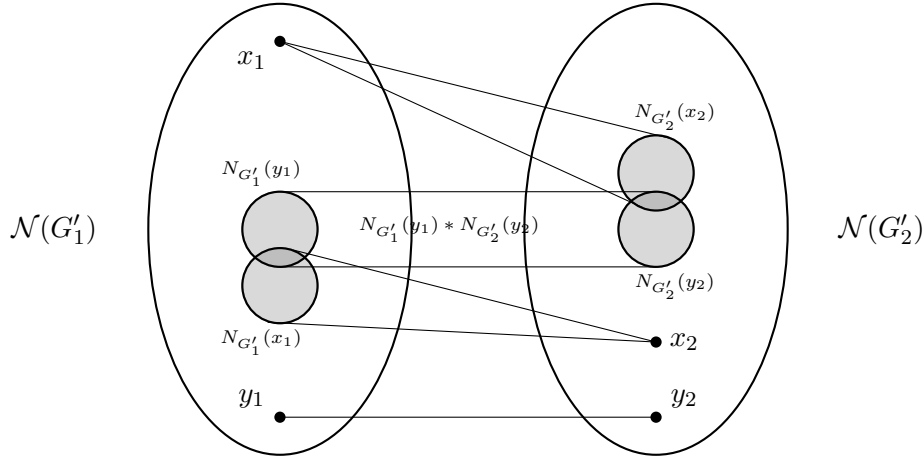
\begin{figure}
\begin{center}
\begin{tikzpicture} [scale = 0.25, every node/.style={}]

\draw[thick, fill = white] (0,0) ellipse (7 and 12);
\draw[thick, fill = white] (20,0) ellipse (7 and 12);
\draw[thick, fill=gray, opacity=.3] (0,-3) circle (2);
\draw[thick, fill=gray, opacity=.3] (0,0) circle (2);
\draw[thick, fill=gray, opacity=.3] (20,3) circle (2);
\draw[thick, fill=gray, opacity=.3] (20,0) circle (2);
\draw[thick] (0,-3) circle (2);
\draw[thick] (0,0) circle (2);
\draw[thick] (20,3) circle (2);
\draw[thick] (20,0) circle (2);

\draw[fill = gray] (0,-10) -- (20,-10);
\draw[fill = gray] (0,-2) -- (20,-2);
\draw[fill = gray] (0,2) -- (20,2);
\draw[fill = gray] (0,10) -- (19.3,1.15);
\draw[fill = gray] (0,10) -- (20.7,4.85);
\draw[fill = gray] (20,-6) -- (.7,-1.1);
\draw[fill = gray] (20,-6) -- (-.65,-4.95);

\node at (0,10) [circle,fill,inner sep=1.5pt]{};
\node at (20,-6)[circle,fill,inner sep=1.5pt]{};
\node at (0,-10) [circle,fill,inner sep=1.5pt]{};
\node at (20,-10)[circle,fill,inner sep=1.5pt]{};

\node (a) at (-1.5,-9) {$y_1$};
\node (b) at (21.5,-9) {$y_2$};
\node (c) at (21,-3) {\tiny $N_{G_2'}(y_2) $};
\node (d) at (21,6) {\tiny $N_{G_2'}(x_2) $};
\node (e) at (-12,0) {$\N(G_1')$};
\node (f) at (32,0) {$\N(G_2')$};
\node (g) at (-1.5,9) {$x_1$};
\node (h) at (21.5,-6) {$x_2$};
\node (i) at (-1,3) {\tiny $N_{G_1'}(y_1) $};
\node (j) at (-1,-6) {\tiny $N_{G_1'}(x_1)$};
\node (k) at (9,0) {\tiny $N_{G_1'}(y_1) \ast N_{G_2'}(y_2) $};
\end{tikzpicture}
\end{center}
\caption{Building $\N(G_1\haj G_2)$ from $\N(G_1)$ and $\N(G_2)$ as in the proof of Theorem~\ref{thm:hajostop}.}
\label{fig:hajostop}
\end{figure}

%%%%%%%%%%%%%%%%%%%%%%

%\vspace{5mm}

 First, consider the effect of identifying $y_1$ and $y_2$. Since $y_1 \in \N(G_1')$ and $y_2 \in \N(G_2')$ are elements in separate connected components, merging the two vertices in the disjoint union $\N(G_1')\uplus \N(G_2')$ is homotopy equivalent to attaching a $1$-cell between $y_1$ and $y_2$, as shown in Figure~\ref{fig:hajostop}. 
In addition, merging $y_1$ and $y_2$ has the effect of joining the faces $N_{G_1'}(y_1)$ and $N_{G_2'}(y_2)$ in $\N(G_1')\uplus \N(G_2')$, which adds $N_{G_1'}(y_1) * N_{G_2'}(y_2)$ to $\N(G_1')\uplus \N(G_2')$. 
Since $N_{G_1'}(y_1) \cap N_{G_2'}(y_2) = \emptyset$, this join connects a face of $\N(G_1')$ to a face of $\N(G_2')$ through a contractible join which is homotopy equivalent to attaching a $1$-cell between  the contractions of $N_{G_1'}(y_1)$ and $N_{G_2'}(y_2)$, respectively. 

 Now, let's consider the effect of adding edge $(x_1,x_2)$. Notice $N_{G_1'}(x_1) \in \mathcal{N}(G_1')$ and $x_2 \in \mathcal{N}(G_2')$, so adding the edge $(x_1,x_2)$ has the effect of coning $N_{G_1'}(x_1)$ over $x_2$. 
Since $N_{G_1'}(x_1)$ is contractible, this is homotopy equivalent to attaching a $1$-cell between $x_2$ and the contraction of $N_{G_1'}(x_1)$. The analysis is similar for $N_{G_2'}(x_2)$ and $x_1$, as shown in Figure~\ref{fig:hajostop}.

We turn our attention to the number of connected components for $\mathcal{N}(G_1')$ and $\mathcal{N}(G_2')$. Suppose $\mathcal{N}(G_2')$ is 0-connected (i.e. $\chi(G_2') \geq 3$). To obtain a copy of $S^1$, at least one of the connected components of $\mathcal{N}(G_1)$ must contain at least two elements from the set $\{\{x_1\}, \{y_1\}, N_{G_1'}(x_1), N_{G_1'}(y_1) \}$, which follows from Lemma~\ref{lem:bipartite}  when $\chi(G_1') \geq 2$. So it suffices to consider  $\chi(G_1')=\chi(G_2') =2.$

 Suppose both $\N(G_1')$ and $\N(G_2')$ have two connected components, respectively. Let $\N(G_1') = A \uplus B$ and $\N(G_2') = C \uplus D$.  From Lemma~ \ref{lem:bipartite} we can assume $x_1, y_1 \in A, N_{G_1'}(x_1), N_{G_1'}(y_1) \in B, x_2, y_2 \in C,$ and $N_{G_2'}(x_2), N_{G_2'}(y_2) \in D$. It follows that, up to homotopy equivalence, the Haj\'os construction attaches 1 -simplices between $A$ and $C$, $A$ and $D$, $B$ and $C$, and $B$ and $D$. 
%It is possible for the faces formed by $N_{G_1}(y_1)\smallsetminus \{x_1\}$ and $N_{G_1}(x_1)\smallsetminus \{y_1\}$ to have a nontrivial intersection, in which case the union of the cone over the face formed by $N_{G_1}(x_1)\smallsetminus \{y_1\}$ and the face formed by $N_{G_1}(y_1)\smallsetminus \{x_1\}$ is still contractible.The situation is identical for $G_2$.

Thus, we have found that $\N(G_1\haj G_2)$ is homotopy equivalent to the space obtained by starting with $\N(G_1')\uplus \N(G_2')$, contracting each of the faces formed by $N_{G_1'}(x_1)$, $N_{G_2'}(x_2)$, $N_{G_1'}(y_1)$, and $N_{G_2'}(y_2)$ to a point and then attaching $1$-cells as indicated above.
Note that if any of these faces intersect, then contracting their union leads to the same structure.
Attaching the $1$-cells between $\N(G_1')\uplus \N(G_2')$ as indicated creates a wedge summand of at least one copy of $S^1$, for both (1) and (2).
\end{proof}

%\jnote{We can eliminate ``without a bridge'' (thought I put it in for now) condition from Cor 3.5 and 3.8 for the following reason: Suppose that $\chi(G_1)$ and $\chi(G_2) \geq 4$. If only one has a bridge it would fall into (2) of Theorem 3.4 meaning we would need to end with vertex identification to get that connectivity up. If both are bridges then we would end with a disconnected graph and therefore a disconnected neighborhood complex. Again we would need to end up with a vertex identification. If neither is a bridge then we got ourselves a case (1). In all cases, we need to end with a vertex identification. }

\begin{corollary}\label{cor:hajostop}
If $G$ is a graph such that $G$ is $j$-constructible, $j \geq 4$, and $\N(G)$ is $i$-connected with $i>0$, then any Haj\'os construction of $G$ must end with a vertex identification.
\end{corollary}

\begin{proof}
  Let $G_1$ and $G_2$ be two graphs used for a Haj\'os merge in a Haj\'os construction of $G$ such that $\chi(G_1), \chi(G_2) \geq 4$.
  Let $(x_1,y_1) \in E(G_1)$ and $(x_2, y_2) \in E(G_2)$ be the edges used in the Haj\'os merge.
  If at most one of the edges is a bridge, Theorem~\ref{thm:hajostop} shows that $G_1 \Delta_H G_2$ is not 1-connected.
  If both $(x_1,y_1)$ and $(x_2,y_2)$ are bridges, $G_1 \Delta_H G_2$ is disconnected and hence $\mathcal{N}(G_1 \Delta_H G_2)$ is also disconnected.
  In either situation, we contradict our assumption that $\mathcal{N}(G)$ is $i$-connected with $i >0$, thus any Haj\'os constuction of $G$ must end with a vertex identification. 
\end{proof}

Corollary~\ref{cor:hajostop} provides a major restriction on Haj\'os constructions of graphs having highly-connected neighborhood complexes.
However, we can provide even stronger restrictions. While vertex identifications are less well-behaved in general than Haj\'os merges, they have a predictable topological effect when the identified vertices are far apart in the graph.
Suppose $v$ and $w$ are vertices in a graph $G$.
Recall that $\dd(v,w)$ denotes the minimum number of edges in a path from $v$ to $w$, and $\dd(v,w)=\infty$ if $v$ and $w$ are not in the same component of $G$.

\begin{lemma} \label{lem:distancefive}
Let $G$ be a graph with vertices $v$ and $w$ such that $\dd(v,w) \geq 5$ and $G' := \vid(G,[v,w])$ with the resulting identified vertex denoted $vw$. 
If $\sigma \subseteq N_{G'}(vw)$ with $\sigma \cap N_G(v) \neq \emptyset$ and  $\sigma \cap N_G(w) \neq \emptyset$, then $\sigma \notin \N(G)$. 
\end{lemma}

\begin{proof}
Suppose that $\sigma \in \N(G)$. 
Let $N_G(v) \cap \sigma := \{v_1,...v_m\}$ and $N_G(w) \cap \sigma := \{w_1,.., w_\ell\}$. 
Since $N_G(v) \cap N_G(w) = \emptyset$, there exists a vertex $x$ such that $\{v_1,..., v_m,w_1,...,w_\ell\} \cup \sigma \subseteq N(x)$. 
In particular, there exists a vertex $v_i \in \{v_1,...,v_m\}$ such that $\dd(v_i, w_j) = 2$ for some $w_j \in \{w_1,...,w_\ell\}$. 
This implies $\dd(v,w) \leq \dd(v_i,w_j) + \dd(v_j,v) + \dd(w_j, w) = 2 + 1 + 1 = 4$, a contradiction to $\dd(v,w) \geq 5$. 
\end{proof}

\begin{theorem}\label{thm:vid5top}
Let $G$ be a connected graph that is not bipartite with $v,w \in V(G)$ such that $\dd(v,w) \geq 5$ and $G' := \vid(G,[v,w])$ with the resulting identified vertex denoted $vw$. 
Then $\N(G') \simeq \N(G) \bigvee S^1 \bigvee S^1$. 
\end{theorem}

\begin{proof}
By Proposition~\ref{prop:connected}, $\N(G)$ consists of one path-connected component, so the simplices $N_G(v), N_G(w), \{v\},$ and $\{w\}$ are all in the same path-connected component. 
Further, since $\dd(v,w) \geq 5$, we have that the sets $N_G(v), N_G(w),  \{v\}, \textrm{and} \{w\}$ are pairwise disjoint.
The operation of identifying $v$ and $w$ in $G$ adds to $\N(G)$ the join of $N_G(v) * N_G(w)$ and identifies $v$ and $w$ in $\N(G)$.
By Lemma~\ref{lem:distancefive}, no faces of $N_G(v) * N_G(w)$ other than $N_G(v)$ and $N_G(w)$ are present in $\N(G)$. 
Therefore, the join is homotopy equivalent to attaching a $1$-cell between the contractions of the facets given by $N_G(v)$ and $N_G(w)$, respectively. 
In addition, the identification of the simplices $\{v\}$ and $\{w\}$ is homotopy equivalent to attaching a $1$-cell between $\{v\}$ and $\{w\}$, as shown in Figure~\ref{fig:vid5top}.
Hence, $\N(G') \simeq \N(G) \bigvee S^1 \bigvee S^1$. 

\end{proof}

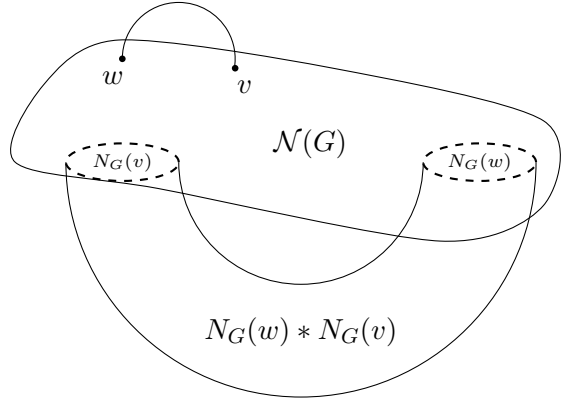
\begin{figure}
\begin{center}
\begin{tikzpicture}[scale = 1.25, every node/.style={}]
\draw [] plot [smooth cycle] %
    coordinates {(-1.14,-1)(-0.84, -.18) (-0.04, 0.3) (2.24, 0) %
    (4.48, -0.56) (4.48, -1.46) (3.38,-1.84)(0.38, -1.28)};

\fill[] (0,0.1) circle (1pt); 
\fill[] (1.2, 0) circle (1pt); 

\draw [thick, dashed] (3.8,-1) ellipse (6mm and 2mm);
%\draw[thick,dashed] (3.8,-1) circle (4mm);
\draw[thick,dashed] (0,-1) ellipse (6mm and 2mm);

\node at (-0.1, -0.1) {$w$};
\node at (1.3, -0.2) {$v$}; 

\node at (2, -.8) {$\N(G)$};

\node at (3.8, -1) {\tiny $N_G(w)$};
\node at (0, -1) {\tiny $N_G(v)$};

\draw[black] (0, 0.1) arc (180: -10:0.6);
\draw[black] (4.4, -1) arc(0: -180: 2.5); 
\draw[black] (3.2, -1) arc (0: -180: 1.3); 

\node at (1.9, -2.8) {$N_G(w) * N_G(v)$};
\end{tikzpicture}
\end{center}
\caption{Through an identification of $v$ and $w$ in $G$ where $\dd(v,w) \geq 5$, we add a $1$-cell $\{v,w\}$ and the join $N(w) *N(v)$ to $\N(G)$ to obtain a space homotopy equivalent to $\N(\vid(G,[v,w]))$.}
\label{fig:vid5top}
\end{figure}

\begin{corollary}\label{cor:vid5top}
If $G$ is a $j$-chromatic graph with $j \geq 4$ such that $\N(G)$ is $i$-connected with $i>0$, then any Haj\'os construction of $G$ must end with a vertex identification $\vid(H,[v,w])$ where $v$ and $w$ satisfy $\dd(v,w)\leq 4$.
\end{corollary}

%%%%%%%%%%%%%%%%%%%%%%%%%%%%%%%%%%%%%%%%%%%%%%%%%%%%%%%%%%%%%%%%%%%%%%%%%%%%%%%%

\section{Identifications of Vertices at Short Distances}\label{sec:shortdistance}

Corollary~\ref{cor:vid5top} demonstrates the importance in Haj\'os-type constructions of identifications of pairs of vertices at distance strictly less than five.
This importance is further emphasized when comparing Corollary~\ref{cor:vid5top} to Theorem~\ref{thm:kahle} regarding the behavior of neighborhood complexes of random graphs.
Since the typical Haj\'os merge or vertex identification at distance strictly greater than four results in a wedge summand of $S^1$ for $\N(G)$, in order to produce a Haj\'os-type construction of $G$ where $\N(G)$ is $i$-connected for $i>0$, it must be the case that vertex identifications at short distances remove these wedge summands, and thus reduce the first homology group of $\N(G)$.
Therefore, we are motivated in this section to consider the effect of such vertex identifications on these wedge summands and on the first Betti number of $\N(G)$, i.e. the rank of the first homology group of $\N(G)$.
We begin with the following lemma.

\begin{lemma}\label{lem:path}
If $G$ is a graph with two vertices $v$ and $w$ such that $p$ is a path from $v$ to $w$ of length four, then $N_G(v)$ and $N_G(w)$ are connected via an edge between two of their respective vertices.
\end{lemma}

\begin{proof}
  Let $p = (v, x_1, x_2, x_3, w)$.
  Then, $x_1 \in N_G(v)$ and $x_3 \in N_G(w)$.
  In addition, $x_1, x_3 \in N_G(x_2)$.
  Therefore,  $N_G(v)$ and $N_G(w)$ are connected through the edge $(x_1, x_3)$ in $\N(G)$.  
\end{proof}

\begin{theorem}\label{thm:d4rankdrop}
  Let $G$ be a graph with vertices $v$ and $w$ such that $\dd(v,w)\geq 3$ and there exists a path of length four from $v$ to $w$. Let $G' := \vid(G, [v,w])$. Then
  \[
    \mathrm{rank}(\widetilde{H}_1(\mathcal{N}(G'))) \leq \mathrm{rank}(\widetilde{H}_1(\mathcal{N}(G)))\, .
  \]
\end{theorem}

\begin{proof}
  Recall that the \emph{link} of a vertex $v$ in a simplicial complex $X$ is $\link_X(v):=\{\sigma\in X:v\notin \sigma, \{v\}\cup \sigma \in X\}$.
  Further, the \emph{star} of a vertex $v$ in $X$ is $\cstar_X(v):=\{\sigma\in X:v\in \sigma\}$.
  By Lemma~\ref{lem:path}, $N_G(v)$ and $N_G(w)$ are connected via an edge in $\N(G)$, though they are disjoint faces in $\N(G)$.
  Since $\dd(v,w)\geq 3$, both $v$ and $w$ are disjoint from these neighborhoods.

  Through the identification of $v$ and $w$, there is a two-step alteration of the neighborhood complex $\N(G)$ that produces $\N(G')$:
\begin{enumerate}
\item $N_{G}(v)$ and $N_{G}(w)$ are joined to become the face $\tau$ with vertex set $N_G(v) \cup N_G(w)$; 
\item any faces containing $v$ or $w$ are deleted, while a new vertex $vw$ is created with
  \[
    \link_{\N(G')}(vw):=\link_{\N(G)}(v)\cup\link_{\N(G)}(w) \, .
  \]
\end{enumerate}

We first address step (1): Consider the complex $\mathcal{N}(G) \cup \tau$.
Since $\tau$ is contractible and thus has zero homology, by the Mayer-Vietoris sequence we obtain: 
\[
  \cdots\rightarrow  \widetilde{H_1}(\mathcal{N}(G) \cap \tau) \rightarrow \widetilde{H}_1(\mathcal{N}(G)) \rightarrow \widetilde{H_1}(\mathcal{N}(G) \cup \tau) \rightarrow \widetilde{H_0}(\mathcal{N}(G) \cap \tau) \rightarrow \cdots
\]
Combining the fact that $\mathcal{N}(G) \cap \tau$ is a subcomplex of $N_G(v)\ast N_G(w)$ with Lemma~\ref{lem:path}, it follows that $\mathcal{N}(G) \cap \tau$ is connected.
Hence, $\widetilde{H_0}(\mathcal{N}(G) \cap \tau) = 0$.
So, $\widetilde{H_1}(\mathcal{N}(G))$ surjects onto $\widetilde{H_1}(\mathcal{N}(G) \cup \tau)$, implying that
\[
  \mathrm{rank}(\widetilde{H_1}(\mathcal{N}(G))) \geq \mathrm{rank}(\widetilde{H_1}(\mathcal{N}(G) \cup \tau)) \, .
\]

We next address step (2), in which we go from $\mathcal{N}(G) \cup \tau$ to $\mathcal{N}(G')$:
Since $\cstar_{\N(G')}(vw)$ is a cone, it is contractible, and thus
\begin{align*}
  \N(G')=\N(\vid(G,[v,w]))& \simeq \N(\vid(G,[v,w]))/\cstar_{\N(G')}(vw) \\
  & \cong (\N(G)\cup \tau)/(\cstar_{\N(G)\cup \tau}(v)\cup \cstar_{\N(G)\cup \tau}(w)) 
\end{align*}
where the homeomorphism of the latter two quotients follows from the fact that $\link_{\N(G')}(vw):=\link_{\N(G)}(v)\cup\link_{\N(G)}(w)$ and the links of $v$ in both $\N(G)$ and $\N(G)\cup \tau$ are the same, and similarly for $w$.
Hence, we have
\[
\widetilde{H}_1(\N(G'))\cong \widetilde{H}_1((\N(G)\cup \tau)/(\cstar_{\N(G)\cup \tau}(v)\cup \cstar_{\N(G)\cup \tau}(w))) \, .
\]
It follows from the existence of a path of length four from $v$ to $w$ that $\cstar_{\N(G)\cup \tau}(v)\cup \cstar_{\N(G)\cup \tau}(w)$ is connected, and combining this with the long exact sequence for the pair $(\N(G)\cup \tau,\cstar_{\N(G)\cup \tau}(v)\cup \cstar_{\N(G)\cup \tau}(w))$ we obtain
\begin{align*}
  \cdots\rightarrow  &\widetilde{H_1}(\cstar_{\N(G)\cup \tau}(v)\cup \cstar_{\N(G)\cup \tau}(w)) \rightarrow \widetilde{H}_1(\mathcal{N}(G) \cup \tau) \rightarrow \widetilde{H_1}((\mathcal{N}(G) \cup \tau)/(\cstar_{\N(G)\cup \tau}(v)\cup \cstar_{\N(G)\cup \tau}(w))) \\
  & \rightarrow \widetilde{H_0}(\cstar_{\N(G)\cup \tau}(v)\cup \cstar_{\N(G)\cup \tau}(w)) = 0 \, .
\end{align*}

So,  $\widetilde{H_1}(\mathcal{N}(G) \cup \tau)$ surjects onto $\widetilde{H_1}(\mathcal{N}(G'))$, implying
\[
  \mathrm{rank}(\widetilde{H_1}(\mathcal{N}(G))) \geq \mathrm{rank}(\widetilde{H_1}(\mathcal{N}(G) \cup \tau)) \geq \mathrm{rank}(\widetilde{H_1}(\mathcal{N}(G'))) \, .
\]
\end{proof}

\begin{theorem}\label{thm:d2rankdrop}
  Let $G$ be a graph with vertices $v$ and $w$.
  Define $U$ to be the set of common neighbors of $v$ and $w$, $A:=N_G(v)\smallsetminus U$, and $B:=N_G(w)\smallsetminus U$.
  Suppose that
  \begin{enumerate}
  \item $\dd(v,w)=2$,
  \item there is no path of length $3$ from $v$ to $w$ in $G$, and
  \item $\left(\bigcup_{a\in A}N_G(a)\right)\cap \left(\bigcup_{b\in B}N_G(b)\right)=\emptyset$.
  \end{enumerate}
  If $G' := \vid(G, [v,w])$, then
  \[
    \mathrm{rank}(\widetilde{H}_1(\mathcal{N}(G'))) \leq \mathrm{rank}(\widetilde{H}_1(\mathcal{N}(G)))\, .
  \]
\end{theorem}

\begin{proof}
  Consider the disjoint union
  \[
    V(G) = \{v,w\}\uplus U\uplus A\uplus B\uplus X
  \]
  where $U$ is the set of common neighbors of $v$ and $w$, $A:=N_G(v)\smallsetminus U$, $B:=N_G(w)\smallsetminus U$, and $X$ contains the remaining vertices in $G$.
  Since $G$ contains no path of length $3$ from $v$ to $w$, the following must be true:
  \begin{itemize}
  \item there is no edge between any two elements of $U$;
  \item $v$ is adjacent to every element of $U$ and $A$;
  \item $w$ is adjacent to every element of $U$ and $B$;
  \item there is no edge between $U$ and $A$, and similarly for $U$ and $B$;
  \item there are no restrictions on the structure of edges within $A$, $B$, or $X$ or the edges connecting elements of $X$ with elements of $A$, $U$, or $B$, respectively.
  \end{itemize}

Note that since $v$ and $w$ share a common neighbor, we have $N_G(v)\cap N_G(w)=U\neq \emptyset$.
Through the identification of $v$ and $w$, there are two changes of the neighborhood complex $\N(G)$ that produce $\N(G')$.
First, the face $\tau:= N_G(v) \cup N_G(w)$ is added to the complex --- note that since $N_G(v)\cap N_G(w)\neq \emptyset$ in this case, $\tau$ is not the join of these two simplices.
Also note that $\tau= U\cup A\cup B$, so $\tau \cap X = \emptyset$.
We can make an identical argument to that given in the first step of Theorem~\ref{thm:d4rankdrop} to show that $\mathrm{rank}(\widetilde{H_1}(\mathcal{N}(G))) \geq \mathrm{rank}(\widetilde{H_1}(\mathcal{N}(G) \cup \tau)),$ except that we argue $\mathcal{N}(G) \cap \tau$ is connected as a result of the property that $N_G(v)\cap N_G(w)\neq \emptyset$.

For the second step, we will show that $\N(G) \cup \tau$ is homotopy equivalent to $\N(G')$, and thus
\[
 \widetilde{H}_1(\N(G')) \cong \widetilde{H}_1(\N(G)\cup \tau)\, ,
\]
from which our result will follow.
Because we have assumed that $\left(\bigcup_{a\in A}N_G(a)\right)\cap \left(\bigcup_{b\in B}N_G(b)\right)=\emptyset$, in $\N(G)\cup \tau$ the edge $(v,w)$ is contained in only $\bigcup_{u\in U}N_G(u)$, while also
\[
v \in \left(\bigcup_{a\in A}N_G(a)\right) \cup\left( \bigcup_{u\in U}N_G(u)\right) \text{ and } w \in \left(\bigcup_{b\in B}N_G(b)\right) \cup \left(\bigcup_{u\in U}N_G(u)\right)\, .
\]
The subcomplex $\bigcup_{u\in U}N_G(u)\subseteq \N(G)\cup \tau$ therefore has the structure of the join of $(v,w)$ with the subcomplex of $\bigcup_{u\in U}N_G(u)$ induced by $X$; within this subcomplex, we make a discrete Morse matching by pairing any cell $\sigma$ containing $v$ but not $w$ with $\sigma\cup\{w\}$.
By Theorem~\ref{thm:maindiscretemorse}, the resulting space, which we denote $\mathcal{X}$, is homotopy equivalent to $\N(G)\cup\tau$.
The topological deformation resulting from this matching has the effect of contracting the joined subcomplex along the edge $(v,w)$, and thus $\mathcal{X}$ is a simplicial complex where $\bigcup_{b\in B}N_G(b)$ has been added to the link of $v$.
This is precisely the description of the link of $vw$ in $\N(G')$, which completes the proof.
\end{proof}

%%%%%%%%%%%%%%%%%%%%%%%%%%%%%%%%%%%

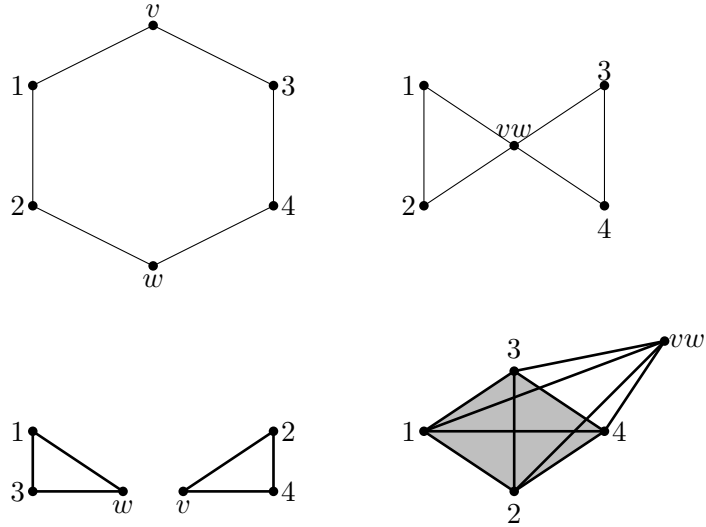
\begin{figure}[ht]
\begin{center}
\begin{tikzpicture}[scale = 0.2, every node/.style={}] 

\begin{scope}[yshift={7cm}]
%bowtie graph
\fill[] (8,4) circle (9pt); 
\fill[] (2,0) circle (9pt);
\fill[] (2,8) circle (9pt);
\fill[] (14,0) circle (9pt);
\fill[] (14,8) circle (9pt);

%hexagon graph
\fill[] (-8,0) circle (9pt);
\fill[] (-8,8) circle (9pt);
\fill[] (-16,-4) circle (9pt);
\fill[] (-24,0) circle (9pt);
\fill[] (-24,8) circle (9pt);
\fill[] (-16,12) circle (9pt);

%hexagon
\node (b) at (-7,0) {$4$};
\node (d) at (-7,8) {$3$}; 
\node (e) at (-16,13) {$v$}; 
\node (e) at (-16,-5) {$w$}; 
\node (g) at (-25,0) {$2$};
\node (h) at (-25,8) {$1$};

%bowtie graph
\draw[fill = gray] (2,0) -- (2,8);
\draw[fill = gray] (2,8) -- (8,4);
\draw[fill = gray] (2,0) -- (8,4);
\draw[fill = gray] (8,4) -- (14,0);
\draw[fill = gray] (8,4) -- (14,8);
\draw[fill = gray] (14,8) -- (14,0);

%hexagon graph
\draw[fill = gray] (-8,0) -- (-8,8);
\draw[fill = gray] (-8,8) -- (-16,12);
\draw[fill = gray] (-16,12) -- (-24,8);
\draw[fill = gray] (-24,8) -- (-24,0);
\draw[fill = gray] (-24,0) -- (-16,-4);
\draw[fill = gray] (-16,-4) -- (-8,0);

%bowtie graph
\node (a) at (1,0) {$2$};
\node (c) at (1,8) {$1$}; 
\node (f) at (8,5) {$vw$}; 
\node (i) at (14,9) {$3$};
\node (j) at (14,-1.5) {$4$};
 
\end{scope}
 %hexagon complex 
 %\filldraw[draw=black,fill=lightgray] (-8,-8) --
(-8,-12) -- (-14,-12);
\draw[line width=1pt] (-8,-8) --
(-8,-12);
\draw[line width=1pt] (-8,-12) --
(-14,-12);
\draw[line width=1pt] (-8,-8) --
(-14,-12);
 %\filldraw[draw=black,fill=lightgray] (-24,-8) --
(-24,-12) -- (-18,-12);
\draw[line width=1pt] (-24,-8) --
(-24,-12);
\draw[line width=1pt] (-24,-12) --
(-18,-12);
\draw[line width=1pt] (-24,-8) --
(-18,-12);

\fill[] (-8,-8) circle (9pt); 
\fill[] (-8,-12) circle (9pt); 
\fill[] (-14,-12) circle (9pt); 
\fill[] (-24,-8) circle (9pt); 
\fill[] (-24,-12) circle (9pt); 
\fill[] (-18,-12) circle (9pt); 

\node (k) at (-7,-8) {$2$};
\node (l) at (-7,-12) {$4$}; 
\node (m) at (-14,-13) {$v$}; 
\node (n) at (-18,-13) {$w$};
\node (o) at (-25,-8) {$1$};
\node (p) at (-25,-12) {$3$};

%bowtie complex 
\filldraw[draw=black,fill=lightgray] (2,-8) --
(8,-12) -- (14,-8)--(8, -4);
\draw[line width=1pt] (2,-8) --
(8,-12);
\draw[line width=1pt] (14,-8) --
(8,-12);
\draw[line width=1pt] (2,-8) --
(14, -8);
\draw[line width=1pt] (8,-12) --
(8,-4);
\draw[line width=1pt] (14,-8) --
(8,-4);
\draw[line width=1pt] (2,-8) --
(8,-4);

\draw[line width=1pt] (2,-8) --
(18, -2);
\draw[line width=1pt] (8,-12) --
(18,-2);
\draw[line width=1pt] (14,-8) --
(18,-2);
\draw[line width=1pt] (18,-2) --
(8,-4);

\fill[] (18,-2) circle (9pt); 
\fill[] (2,-8) circle (9pt); 
\fill[] (8,-12) circle (9pt); 
\fill[] (14,-8) circle (9pt); 
\fill[] (8,-4) circle (9pt); 

\node (q) at (19.5,-2) {$vw$};
\node (r) at (1,-8) {$1$}; 
\node (s) at (8,-13.5) {$2$}; 
\node (t) at (8,-2.5) {$3$};
\node (u) at (15,-8) {$4$};

\end{tikzpicture}
\end{center}
\caption{Top left: Graph $G$ with $\dd(v,w)= 3$; Top right: Graph $G':=\textrm{vid}(G,(v,w))$; Bottom left: $\mathcal{N}(G)$; and Bottom right: $\mathcal{N}(G')$. Observe that $\textrm{rank}(\widetilde{H}_1(\mathcal{N}(G)))=2$ and $\textrm{rank}(\widetilde{H}_1(\mathcal{N}(G'))) = 3$. }
\label{fig:example}
\end{figure}

%%%%%%%%%%%%%%%%%%%%%%%%%%%%%%%%%%%%%%%%%%%%%%%%%%%%%%%

\begin{remark}
Without a path of length four, when $\dd(v,w) = 3$ we cannot conclude in general that $\textrm{rank}(\widetilde{H}_1(\mathcal{N}(G'))) \leq \textrm{rank}(\widetilde{H}_1(\mathcal{N}(G)))$, since our Mayer-Vietoris argument no longer holds.
As an example, consider the graph $G$ depicted in Figure~\ref{fig:example} and perform a vertex identification on $v$ and $w$.
An analogous consideration when $G$ is a cycle on $9$ vertices yields similar results without increasing $\chi(G)$.
\end{remark}

We next demonstrate that a single vertex identification can result in an arbitrarily large decrease in the size of the first homology of $\N(G)$.
Further, this decrease can arise from the elimination of an arbitrarily large number of $S^1$ wedge summands.

\begin{theorem}\label{thm:bettidrop}
  For every $n\geq 5$, there exists a graph $G_n$ and a single vertex identification in $G_n$ resulting in a graph $G_n'$ such that:
  \begin{itemize}
  \item $\N(G_n)$ is a wedge of $2n+5$ circles, and
    \item $\N(G_n')$ has trivial first homology.
  \end{itemize}
\end{theorem}

\begin{proof}
The theorem follows from Propositions~\ref{prop:circles} and~\ref{prop:spheres}.
\end{proof}

We define our desired $G_n$ as follows.

\begin{definition}\label{def:bettidropgraph}
  Let $G_n$ be the graph with vertex and edge sets defined as follows:
  \[
    V(G) = \{X,Y,Z \}\cup\{iA, iB, iC :1 \leq i \leq n\}
  \]
  and
  \begin{align*}
    E(G) :=& \hspace{4.3mm} \{ (X,Y), (Y,Z), (1A, X), (nC, X)\}\\
          & \cup \{ (iB, X), (iA, iB), (iB,iC), (iC,iA): 1 \leq i \leq n\}\\
          & \cup \{(jC, Z),  (jC, (j+1)A) :  1 \leq j \leq n-1 \}\\
    & \cup \{(kA, Z): 2 \leq k \leq n\}
  \end{align*}
  Define $G_n' := \vid(G_n, [X,Z])$.
\end{definition}

$G_3$ and $G_3'$ are shown in Figure~\ref{fig:drops}.
It is straightforward to argue that $\chi(G_n) = 4$ for all $n$, and it is also possible to provide an explicit Haj\'os construction of $G_n$ starting with iterated Haj\'os merges using $n$ copies of $K_4$, followed by a sequence of vertex identifications.

%%%%%%%%%%%%%%%%%%%%%%%%%%%%%%%%%%%

\begin{figure}[h]

\begin{center}
\begin{tikzpicture}[scale = 0.2, every node/.style={}]

\fill[] (-23,0) circle (9pt); 
\fill[] (-19.5,0) circle (9pt);
\fill[] (-14,0) circle (9pt);
\fill[] (-12,0) circle (9pt);
\fill[] (-9,0) circle (9pt);
\fill[] (-7,0) circle (9pt);
\fill[] (-1,0) circle (9pt);

\fill[] (-15.5,2) circle (9pt);
\fill[] (-10.5,2) circle (9pt);
\fill[] (-5.5,2) circle (9pt);
\fill[] (-10.5,7) circle (9pt);
\fill[] (-10.5,-5) circle (9pt);

\fill[] (20,0) circle (9pt);
\fill[] (14,0) circle (9pt);
\fill[] (12,0) circle (9pt);
\fill[] (9,0) circle (9pt);
\fill[] (7,0) circle (9pt);
\fill[] (1,0) circle (9pt);

\fill[] (15.5,2) circle (9pt);
\fill[] (10.5,2) circle (9pt);
\fill[] (5.5,2) circle (9pt);
\fill[] (10.5,7) circle (9pt);
\fill[] (10.5,13) circle (9pt);

\draw[line width=1pt] (-23,0) --
(-10.5,7);
\draw[line width=1pt] (-23,0) --
(-10.5,-5);
\draw[line width=1pt] (-19.5,0) --
(-10.5,7);
\draw[line width=1pt] (-19.5,0) --
(-15.5,2);
\draw[line width=1pt] (-15.5,2) --
(-14,0);
\draw[line width=1pt] (-14,0) --
(-19.5,0);
\draw[line width=1pt] (-14,0) --
(-12,0);
\draw[line width=1pt] (-15.5,2) --
(-10.5,7);
\draw[line width=1pt] (-10.5,2) --
(-10.5,7);
\draw[line width=1pt] (-5.5,2) --
(-10.5,7);
\draw[line width=1pt] (-1,0) --
(-10.5,7);
\draw[line width=1pt] (-14,0) --
(-10.5,-5);
\draw[line width=1pt] (-12,0) --
(-10.5,-5);
\draw[line width=1pt] (-9,0) --
(-10.5,-5);
\draw[line width=1pt] (-7,0) --
(-10.5,-5);
\draw[line width=1pt] (-12,0) --
(-9,0);
\draw[line width=1pt] (-12,0) --
(-10.5,2);
\draw[line width=1pt] (-9,0) --
(-10.5,2);
\draw[line width=1pt] (-7,0) --
(-9,0);
\draw[line width=1pt] (-7,0) --
(-5.5,2);
\draw[line width=1pt] (-7,0) --
(-1,0);
\draw[line width=1pt] (-1,0) --
(-5.5,2);

%%%%%%%%%%POSITIVE SIDE%%%%%%%%
\draw[line width=1pt] (10.5,13) --
(10.5,7);
\draw[line width=1pt] (20,0) --
(10.5,7);
\draw[line width=1pt] (20,0) --
(15.5,2);
\draw[line width=1pt] (15.5,2) --
(14,0);
\draw[line width=1pt] (14,0) --
(20,0);
\draw[line width=1pt] (14,0) --
(12,0);
\draw[line width=1pt] (15.5,2) --
(10.5,7);
\draw[line width=1pt] (10.5,2) --
(10.5,7);
\draw[line width=1pt] (5.5,2) --
(10.5,7);
\draw[line width=1pt] (1,0) --
(10.5,7);
\draw[line width=1pt] (7,0) --
(10.5,7);
\draw[line width=1pt] (9,0) --
(10.5,7);
\draw[line width=1pt] (12,0) --
(10.5,7);
\draw[line width=1pt] (14,0) --
(10.5,7);

\draw[line width=1pt] (12,0) --
(9,0);
\draw[line width=1pt] (12,0) --
(10.5,2);
\draw[line width=1pt] (9,0) --
(10.5,2);
\draw[line width=1pt] (7,0) --
(9,0);
\draw[line width=1pt] (7,0) --
(5.5,2);
\draw[line width=1pt] (7,0) --
(1,0);
\draw[line width=1pt] (1,0) --
(5.5,2);

%%%%%%%%%%%%%NODE LABELS%%%%%%%%%%%%
\node (a) at (-24,0) {\tiny$Y$};
\node (b) at (-18.5,-0.7) {\tiny$1A$}; 
\node (c) at (-10.5,8) {\tiny$X$}; 
\node (d) at (-10.5,-6) {\tiny$Z$};
\node (e) at (-10.8,-0.7) {\tiny$2A$};
\node (f) at (-6,-0.7) {\tiny$3A$};

\node (g) at (-14,2) {\tiny$1B$};
\node (h) at (-9.2,2) {\tiny$2B$};
\node (i) at (-6.6,2) {\tiny$3B$};

\node (j) at (-1,-0.7) {\tiny$3C$};
\node (j) at (-15,-0.7) {\tiny$1C$};
\node (j) at (-8,0.7) {\tiny$2C$};

\node (b) at (20,-0.9) {\tiny$3C$}; 
\node (c) at (12,8) {\tiny$XZ$}; 
\node (d) at (12,13) {\tiny$Y$};
\node (e) at (12,-0.9) {\tiny$2C$};
\node (f) at (7,-0.9) {\tiny$1C$};

\node (g) at (14,-0.9) {\tiny$3A$};
\node (h) at (9,-0.9) {\tiny$2A$};
\node (i) at (5,0.8) {\tiny$1B$};

\node (j) at (1,-0.9) {\tiny$1A$};
\node (j) at (10.5,0.8) {\tiny$2B$};
\node (j) at (16,0.8) {\tiny$3B$};

\end{tikzpicture}
\end{center}
\caption{$G_3$ is displayed on the left and $G_3'=\textrm{vid}(G_3, [X,Z])$ on the right. $G_n$ along with $G_n'$ defines a family of graphs for which the decrease in the first homology group after one vertex identification can be arbitrarily large. }
\label{fig:drops}
\end{figure}
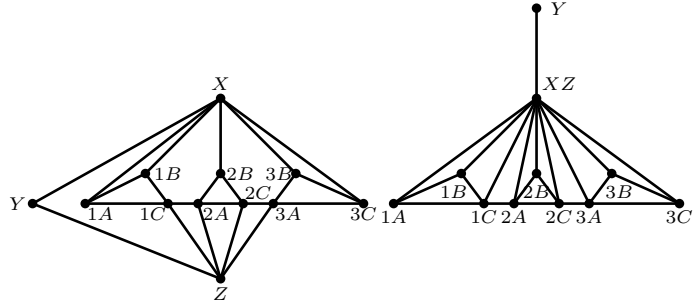

%%%%%%%%%%%%%%%%%%%%%%%%%%%%%%%%%%%

\begin{prop}\label{prop:circles}
For $n \geq 5$, $\N(G_n) \simeq  \bigvee \limits_{2n+5}S^1$. 
\end{prop}

\begin{proof}
  Let $P_n$ denote the face poset of $\N(G_n)$.
  Every vertex $v \in G_n$ generates a subposet isomorphic to a boolean algebra $B_{d}$ containing the subsets of $N_{G_n}(v)$, where $d$ is the cardinality of $N_{G_n}(v)$.
  The neighborhood of $X$ generates a subposet isomorphic to $B_{n+3}$ and the neighborhood of $Z$ generates a subposet isomorphic to $B_{2n-1}$.
  The remaining vertices have degree $2$, $3$, or $4$ giving rise to subposets isomorphic to $B_2$, $B_3$, or $B_4$, respectively.
  To provide notation for this, let $\da(v):= \{p \in P_n \, : \, p \subseteq N_{G_n}(v)\}$ and $\da(v)_{\geq 1} := \{ p \in \da(v) \, : \, |p| \geq 1\}$, where NP stands for \emph{Neighborhood Poset}.
  The strategy for this proof is to make systematic acyclic matchings on parts of the small Boolean algebras, then to quotient the resulting space by the subcomplex of $\N(G_n)$ given by $N_{G_n}(X)\cup N_{G_n}(Z)$, which is contractible.
%In the construction of our acyclic matching on $P_n$ below, it is helpful to consider $P_n$ as a disjoint union of the following sets:
%\begin{itemize}
%\item $\da(X)\cup \da(Z)$
%\item $\da(Y)$
%\item $\da(nB)_{\geq 1}$
%  \item $\da(1B)_{\geq 1}$
%    \item $\da(nC)_{\geq 1}$
%    \item $\da(1C)_{\geq 1}\setminus \{\{1A,1B\}\}$
%    \item $\da(nA)_{\geq 1}\setminus \{\{nB,nC\}\}$
%      \item $\da(jB)_{\geq 1}\smallsetminus \{\{jA,jC\}\}$ for $2\leq j\leq n-1$
%      \item $\da(jA)_{\geq 1}\smallsetminus \{\{(j-1)C,jC\},\{(j-1)C,Z\}$ for $2\leq j\leq n-1$
%      \item $\da(jC)_{\geq 1}\smallsetminus \{\{jA,(j+1)A\},\{jA,Z\},\{jB,Z\}\}$ for $2\leq j\leq n-1$
%\end{itemize}

\begin{figure}
  \begin{tikzpicture}[scale = 0.2, every node/.style={}]

\draw[line width=1pt] (0,-4) --
(0,36);
\draw[line width=1pt] (0,4) --
(-4,8);
\draw[line width=1pt] (0,12) --
(12,24);

\draw[line width=1pt] (12,24) --
(18,30);
\draw[line width=1pt] (4,16) --
(0,24);
\draw[line width=1pt] (8,20) --
(0,32);

\fill[] (0,-4) circle (9pt); 
\fill[] (0,0) circle (9pt);
\fill[] (0,4) circle (9pt);
\fill[] (0,8) circle (9pt);
\fill[] (0,12) circle (9pt);
\fill[] (0,16) circle (9pt);
\fill[] (0,20) circle (9pt);

\fill[] (0,24) circle (9pt); 
\fill[] (0,28) circle (9pt);
\fill[] (0,32) circle (9pt);
\fill[] (0,36) circle (9pt);
\fill[] (4,16) circle (9pt);
\fill[] (8,20) circle (9pt);
\fill[] (12,24) circle (9pt);
\fill[] (18,30) circle (9pt);
\fill[] (-4,8) circle (9pt);

\node (a) at (-1.5,-4) {$\mathcal{O}$};
\node (b) at (-1.5,0) {\tiny$\mbf{1A}$}; 
\node (c) at (1.5,4) {\tiny$\mbf{2A}$}; 
\node (d) at (-5.5,8) {\tiny$\mbf{X}$};
\node (e) at (-1.5,8) {\tiny$\mbf{Z_2}$};
\node (f) at (-1.5,12) {\tiny$\mbf{2C}$};

\node (g) at (-1.5,16) {\tiny$\mbf{Z_3}$};
\node (h) at (-1.5,20) {\tiny$\mbf{3A}$};
\node (i) at (-1.5,24) {\tiny$\mbf{Z_4}$};

\node (j) at (-1.5,28) {\tiny$\mbf{4A}$};
\node (j) at (-1.5,32) {\tiny$\mbf{Z_5}$};
\node (j) at (-1.5,36) {\tiny$\mbf{5A}$};

\node (b) at (5.5,16) {\tiny$\mbf{3C}$}; 
\node (c) at (9.5,20) {\tiny$\mbf{4C}$}; 
\node (d) at (12,26) {\tiny$\mbf{6C}$};
\node (e) at (18,32) {\tiny$\mbf{5C}$};

\end{tikzpicture}
\caption{$Q_6$}
\label{fig:Qnpic}
\end{figure}

  In order to create our systematic acyclic matchings, we will apply Theorem~\ref{thm:patchwork}.
Define $Q_n$ to be the poset on the elements
\[
  \{\mbf{\mathcal{O}, X, Z_2,..., Z_{(n-1)}, 1A, 2A, ..., (n-1)A, 2C, 3C,...., nC}\}
\]
with cover relations given by: 
\begin{itemize}
    \item $\mathcal{O} \prec \mbf{1A} \prec \mbf{2A} \prec \mbf{Z_2} \prec \mbf{2C} \prec \mbf{Z_3} \prec \mbf{3A}$    
    \item $\mathbf{2A} \prec \mbf{X}$
    \item $\mbf{2C} \prec \mbf{3C} \prec \cdots \prec \mbf{(n-2)C} \prec \mbf{nC}\prec \mbf{(n-1)C}$
    \item $\mbf{jA} \prec \mbf{Z_{j+1}} \prec \mbf{(j+1)A}$ for all $3 \leq j \leq n-2$
    \item $\mbf{jC} \prec \mbf{Z_{j+1}}$ for all $3 \leq j \leq (n-2)$ 
\end{itemize}
$Q_6$ is illustrated in Figure~\ref{fig:Qnpic}.
To distinguish target elements in $Q_n$ from the vertices of $G_n$, we write the elements of $Q_n$ in bold.

We next define a poset map $\Gamma_n: P_n \rightarrow Q_n$.
The preimage of an element $\alpha \in Q_n$ is denoted as $\Gamma_n^{-1}(\alpha)$.
We begin by setting $\Gamma_n^{-1}(\mathcal{O})=\da(X) \, \cup \da(Z)$.
No matching will take place on this subposet of $P_n$, and the corresponding subcomplex will remain contractible.
This accounts for both of the large Boolean algebras in $P_n$.
What remains unmapped are a handful of cells in each smaller Boolean algebra, which we must handle individually, leading to a long list of preimages.
For each element in $Q_n$, we define $\Gamma_n^{-1}$ in a manner that makes it straightforward (though tedious) to verify that the resulting map is order preserving.
For the elements in $Q_n$ that are less than or equal to $\mbf{2C}$, we make the following assignments.
Observe that when we remove elements from our neighborhood posets in the mapping below, it is due to the fact that those elements were already mapped in a previous preimage assignment.

\begin{itemize}
\item $\Gamma_n^{-1}(\mbf{1A})= \{\{1A, 1B,2A\}, \{1A, 2A\}, \{1B, 2A\} \}$
    \item $\Gamma_n^{-1}(\mbf{2A})=$ \begin{align*}&( \da(Y) \cup \da(1C) \cup \da(2C) \cup \da(2B)) \\
     & \smallsetminus \{\{1A,1B,2A\}, \{1A,2A\}, \{1B,2A\}, \{1A,1B\}, \{2A,3A\}, 1A, 1B, 2A, 2B, 2C, 3A\}\end{align*}
    \item $\Gamma_n^{-1}(\mbf{X}) = (\da(1B) \cup \da(1A) \cup \da(nB) \cup \da(nC))_{\geq 1}$

    \item $\Gamma_n^{-1}(\mbf{Z_2})= \{\{2B, 2C, Z\}, \{2B,2C\}, \{2C,Z\} \} $ 

    \item $\Gamma_n^{-1}(\mbf{2C})= $ \begin{align*}&(\da(2A) \cup \da(3A))_{\geq 1} \\
      &\smallsetminus \{\{2B,2C,Z\}, \{2B,2C\}, \{2B,Z\}, \{1C,2C\}, \{2C,3C\}, \{2C,Z\} \} \end{align*}
  \end{itemize}

  At this point, all of the elements of $P_n$ contained in $\da(X)$, $\da(Z)$, $\da(Y)$, $\da(1C)$, $\da(1A)$, $\da(2A)$, $\da(3A)$, $\da(nC)$, $\da(1B)$, $\da(nB)$, $\da(2B)$, and $\da(2C)$ have been assigned an image under $\Gamma_n$.
  Next, we map the remaining elements of $P_n$ to elements of $Q_n$ that are above $\mbf{2C}$.
  
  \begin{itemize}
    \item $\Gamma_n^{-1}(\mbf{jC})= \da((j+1)A)_{\geq 1} \smallsetminus \{\{jC,Z\}, \{jC, (j+1)C\}\}$ for each $3 \leq j \leq (n-2)$
    \item $\Gamma_n^{-1}(\mbf{Z_k})=\{\{kA, kB, Z\}, \{kA, kB\}\}$ for each $3 \leq k \leq (n-1)$.
    \item $\Gamma_n^{-1}(\mbf{kA})=$\begin{align*}&(\da(kC) \cup \da(kB))_{\geq1} \\
      &\smallsetminus \{\{kA,kB,Z\}, \{kA,kB\}, \{kB,Z\}, \{kA,Z\}, \{kA, (k+1)A\}\}\end{align*} for each $3 \leq k \leq (n-1)$.
    \item $\Gamma_n^{-1}(\mbf{nC})= \{\{(n-1)C,nB,nC\},\{(n-1)C,nC\},\{(n-1)C,nB\}\}$
    \item $\Gamma_n^{-1}(\mbf{(n-1)C})=$ \begin{align*}&\da(nA)_{\geq 1} \\ &\smallsetminus \{\{(n-1)C,nB,nC\}, \{(n-1)C,nB\}, \{(n-1)C,nC\}, \{(n-1)C,Z\},\{nB,nC\}\} \end{align*}
\end{itemize}

%The map $\Gamma_n$ is defined moving from left to right through the graph as depicted in Figure~\ref{fig:drops}, with an ephasis on the neighborhoods $iC, jA, 1B,$ and $nB$ for $1 \leq i \leq n$ and $2 \leq j \leq n$. $\Gamma_n^{-1}(\mbf{\mathcal{O}})$ is removed from each set because $N(X) \cup N(Z)$ contains $G_n \smallsetminus \{X, Z\}$. Outside of $N(X) \cup N(Z)$, each neighborhood $N((j+1)A)$ intersects $N(jA)$ to the right for $3 \leq j \leq (n-2)$. Similarly, $N(kC)$ intersects $N((k-1)C)$ and $N(kA)$ to the right for $3 \leq k \leq (n-1)$. Finally, $Z_2 \subseteq 2A$ and $Z_k \subseteq kC$ for $3 \leq k \leq (n-1)$. Since each of the sets, $S$, in the domain of $\Gamma$ are defined by taking a downward closed poset and removing a preimage of an element that is less than $\Gamma(S)$, that is also lower order ideal, the map $\Gamma$ is order-preserving. 

To complete our proof, we define a matching on each $\Gamma_n^{-1}(\bf{t})$ for $\bf{t} \in Q_n \smallsetminus \mbf{\mathcal{O}}$.
For each $\sigma\in\Gamma_n^{-1}(\bf{t})$ that does not contain the vertex $t$ (the bold symbol is an element of $Q_n$, the unbolded symbol is a vertex of $G_n$), we pair $\sigma$ with $\sigma \cup \{t\}$ --- if the element is $\mbf{Z_i}$, we use the vertex $Z$ as the unbolded version of $\mbf{Z_i}$.
Because for each of these matchings there is a unique element that is used for the pairing assignment, it is a quick exercise to confirm that these are all acyclic matchings.
Thus, by Theorem~\ref{thm:patchwork}, we have defined an acyclic matching on $P_n$. 
In addition to $\Gamma_n^{-1}(\mbf{\mathcal{O}})$, the corresponding critical (i.e. unmatched) cells in each preimage are:
\begin{itemize}
    \item $\Gamma_n^{-1}(\mbf{1A}) : \{1A, 2A\}$ 
    \item $\Gamma_n^{-1}(\mbf{2A}) : \{2A, 2B\}$, $\{X,Z\}$ 
    \item $\Gamma_n^{-1}(\mbf{X}) : \{1A,X\}$, $\{1C,X\}$, $\{1B,X\}$, $\{nA, X\}$, $\{nB, X\}$, $\{nC,X\}$ 
    \item $\Gamma_n^{-1}(\mbf{Z_2}) : \{2C, Z\}$ 
    \item $\Gamma_n^{-1}(\mbf{2C}) : \{2C,3B\}$ 
    \item $\Gamma_n^{-1}(\mbf{jC}) : \{jC, (j+1)B\}$ for $3 \leq j \leq n-2$, which gives rise to $n-4$ critical 1-cells.  
    \item $\Gamma_n^{-1}(\mbf{kA}) : \{kA, X\}$ for $3\leq k \leq n-1$, which gives rise to $n-3$ critical 1-cells. 
    \item $\Gamma_n^{-1}(\mbf{Z_k})$ has no critical cells for any $3 \leq k \leq (n-1)$ 
    \item $\Gamma_n^{-1}(\mbf{nC}) : \{(n-1)C, nC\}$
    \item $\Gamma_n^{-1}(\mbf{(n-1)C})$ has no critical cells.
\end{itemize}
This gives a total of $2n + 5$ critical 1-cells.
Since $N_{G_n}(X) \cup N_{G_n}(Z)$ forms a contractible subcomplex of the critical complex for this matching, we can contract this subcomplex yielding $\N(G_n) \simeq  \bigvee \limits_{2n+5}S^1$.
\end{proof}

\begin{prop}\label{prop:spheres}
For $n \geq 5$, we have $\N(G_n') \simeq \bigvee\limits_{2n-1}S^2$. 
\end{prop}
\begin{proof}

Let $P_n'$ denote the face poset of $\N(G_n')$. 
Define $Q_n'$ to be the poset on the elements 
\[
\{\mbf{\mathcal{O}'},\mbf{\mathcal{T}}\}\cup\{\mbf{jA, jC}:2 \leq j \leq (n-1)\}
\]
with cover relations given by $ \mbf{\mathcal{O}'} \prec \mbf{2A}$, $\mbf{iA} \prec \mbf{iC} \prec \mbf{(i+1)A}$ for $2 \leq i \leq (n-2)$, and maximal element $\mbf{\mathcal{T}}$.
Thus, $Q_n'$ is totally ordered.
As in the previous proof, it is helpful to think of $P_n'$ as a union of subsets of neighborhoods of the vertices of $G_n'$, adjusted to remove the subsets contained in the neighborhood of $XZ$.
Similarly to our previous proof, our strategy will be to leave the subsets of the neighborhood of $XZ$ unmatched, pair off subsets of the remaining vertices in $\N(G_n')$, and then quotient out the neighborhood of $XZ$ at the end.

We define a poset map $\Gamma_n': P_n' \rightarrow Q_n'$.
As before, let $\da(v):= \{p \in P_n \, : \, p \subseteq N_{G_n}(v)\}$ and $\da(v)_{\geq 1} := \{ p \in \da(v) \, : \, |p| \geq 1\}$.
Begin by setting $(\Gamma_n')^{-1}(\mbf{\mathcal{O}'}) = \da(XZ)$; no matching will take place on this subposet of $P_n'$ and the corresponding subcomplex will remain contractible.
For each of the remaining elements in $Q_n$, we define $(\Gamma_n')^{-1}$ in a manner that makes it straightforward to verify that the resulting map is order-preserving.
We first define
\[
(\Gamma_n')^{-1}(\mbf{2A})=(\da(1C) \cup \da(2B) \cup \da(2C))_{\geq 1} \smallsetminus \da(XZ) \, .
\]
Next, to account for $\da(1A)$, $\da(1B)$, $\da(2A)$, $\da(3A)$, and $\da(3B)$, set
\begin{align*}
(\Gamma_n')^{-1}(\mbf{2C}) =& \{\{1C,2B,2C,XZ\}, \{2C,3B,3C,XZ\}, \{1A,1C,XZ\}, \{1B,1C,XZ\}, \\
                            & \{1C,2B,XZ\}, \{2C,3C,XZ\}, \{2B,2C,XZ\}, \{2C,3B,XZ\}, \{3B,3C,XZ\}, \{3A,3C,XZ\}, \\
 & \{1C,2C,XZ\}, \{1C, XZ\}, \{3B,XZ\}, \{3C,XZ\}\} \, .
\end{align*}

We next make two types of iterated assignments, to account for the neighborhoods of the remaining vertices of $G_n'$, and a final assignment for maximal triangles in $\N(G_n')$.
\begin{itemize}
\item $(\Gamma_n')^{-1}(\mbf{kA})= \{\{kA, kB, (k+1)A, XZ\}, \{kB, (k+1)A, XZ\}, \{kA, (k+1)A, XZ\},  \{(k+1)A, XZ\}, \{kA,kB,XZ\}\}$ for  $3 \leq k \leq n$, where we do not include the four sets in $(\Gamma_n')^{-1}(\mbf{nA})$ containing an element labeled $(n+1)A$, thus $|(\Gamma_n')^{-1}(\mbf{nA})|=1$.
\item $(\Gamma_n')^{-1}(\mbf{kC})= \{\{kC, (k+1)B, (k+1)C, XZ\}, \{(k+1)B, (k+1)C, XZ\}, \{kC, (k+1)C, XZ\}, \\
  \{kC, (k+1)B, XZ\}, \{(k+1)B, XZ\}, \{(k+1)C, XZ\}\}$ for $3 \leq k \leq n$.
      \item $(\Gamma_n')^{-1}(\mbf{\mathcal{T}})=\{\{kA,kC,XZ\}:4\leq k\leq n\}$
      \end{itemize}

All cells in $(\Gamma_n')^{-1}(\mbf{\mathcal{T}})$ are unmatched.
As in the previous proof, for each $\sigma\in\Gamma_n^{-1}(\bf{t})$ that does not contain the vertex $t$ (the bold symbol is an element of $Q_n$, the unbolded symbol is a vertex of $G_n$), we pair $\sigma$ with $\sigma \cup \{t\}$.
For each matching there is a unique element used for pairing and therefore all matchings are acyclic. Thus by Theorem~\ref{thm:patchwork}, we have defined a matching on $P_n'$. Other than the elements in $(\Gamma_n')^{-1}(\mbf{\mathcal{O}})$, our critical cells are: 
\begin{itemize}
    \item $(\Gamma_n')^{-1}(\mbf{2A})$ has no critical cells.
    \item $(\Gamma_n')^{-1}(\mbf{2C}) : \{1A,1C,XZ\}$, $\{1B,1C,XZ\}, \{2B,2C,XZ\}$, $\{3A,3C,XZ\}$
    \item $(\Gamma_n')^{-1}(\mbf{kA}): \{kA, kB, XZ\}$ for $3 \leq k \leq n$, which gives rise to $n-2$ critical 2-cells.
    \item $(\Gamma_n')^{-1}(\mbf{kC})$ has no critical cells for $3 \leq k \leq n$.
    \item $(\Gamma_n')^{-1}(\mbf{\mathcal{T}}) : \{kA,kC,XZ\}$ for $4\leq k\leq n$, which gives rise to $n-3$ critical 2-cells.
\end{itemize}
This gives a total of $2n - 5 + 4 = 2n-1$ critical 2-cells. 
Since $(\Gamma_n')^{-1}(\mbf{\mathcal{O}})$ corresponds to a contractible subcomplex of the critical complex for this matching, we can contract this subcomplex and obtain $\N(G_n') \simeq \bigvee\limits_{2n-1}S^2$. 

\end{proof}

%%%%%%%%%%%%%%%%%%%%%%%%%%%%%%%%%%%%%%%%%%%%%%%%%%%%%%%%%%%%%
\section{Topological Effects of DHGO Compositions}\label{sec:dhgo}

Recall that a graph $G$ is \emph{$k$-extremal} if $G$ is a graph on $n$ vertices that is $k$-chromatic and $k$-critical, i.e. any subgraph of $G$ has chromatic number lower than $k$, and if $G$ has the minimum number of edges possible among such graphs on $n$ vertices.
Kostochka and Yancey observed~\cite{KostochkaYancey} that Ore thought beginning with an extremal graph on at most $2k$ vertices and repeatedly using $K_k$ as $G_2$ in DHGO compositions, defined below, would result in the best possible construction of sparse critical graphs.

\begin{definition}
For a graph $G$ and $u \in V(G)$, a \emph{split} of $u$, denoted $\text{sp}(G,u)$, is a new graph $G''$ with vertex set $V(G'') = V(G) \smallsetminus \{u\} + \{u', u''\}$, where $G \smallsetminus \{u\} \cong G'' \smallsetminus \{u, u''\}$, $N(u') \cup N(u'') = N(u),$ and $N(u') \cap N(u'') = \emptyset$.
\end{definition}
 
\begin{definition}
Let $G_1$ and $G_2$ be graphs such that $x_1 \in V(G_1)$ and $(x_2,y_2) \in E(G_2)$. 
A \emph{DHGO-compostion}, denoted $D(G_1, G_2)$, is a graph that is created by deleting $(x_2,y_2)$, splitting $x_1$ into $x_1'$ and $x_1''$ with positive degree, and identifying $x_1'$ with $x_2$ and $x_1''$ with $y_2$
\end{definition}

It is a straightforward exercise to show that a Haj\'os merge is a special case of a DHGO composition, which motivates us to consider the topological effect on $\N(G)$ of these more general operations.

\begin{lemma} \label{lem:split}
Let $G$ be a connected non-bipartite graph such that $G'':= \text{sp}(G,u)$ is connected bipartite where $N(u') \cup N(u'') = N(u)$ and $N(u') \cap N(u'') = \emptyset$. Then, $\mathcal{N}(G'') = A \uplus B$, where $N_{G''}(u'), u'' \in A$ and $N_{G''}(u''), u' \in B$. 
\end{lemma}

\begin{proof}
Since $\chi(G'') = 2$, $u$ must be part of all odd cycles in $G$. This implies that there is an odd path from $u'$ to $u''$, so  $u'', N(u') \in A$ and  $N_{G''}(u''), u' \in B$.  
If $u'$ and $u''$ are in the same connected component, there is an even path between $u'$ and $u''$, implying there is only one connected component, which is a contradiction to being connected bipartite. 
\end{proof}
 
\begin{theorem}\label{thm:dhgo}
Consider two connected graphs $G_1$ and $G_2$ with $x_1 \in V(G_1)$ and $(x_2,y_2) \in E(G_2)$ such that: 
\begin{enumerate}
\item $\chi(G_1), \chi(G_2) \geq 3$ and $(x_2,y_2)$ is not bridge, and
\item $\chi(G_1) \geq 3$, $\chi(G_2) \geq 4$, and $(x_2, y_2)$ is not a bridge.
\end{enumerate}
Then, $\mathcal{N}(D(G_1, G_2))$ is a wedge of at least one copy of $S^1$ with another space. 
\end{theorem}

\begin{proof}
The result follows analogously to Theorem~\ref{thm:hajostop} with an adaptation arising in the steps of the operation.
Let $G_1':=\textrm{sp}(G_1,x_1)$ and $G_2':=G_2 \smallsetminus (x_2,y_2)$ with the DGHO-composition identifiying $x_1', x_2$ and $x_1'', y_2$. It follows as in Theorem \ref{thm:hajostop}, the effect of identifying $x_1'$ with $x_2$ is homotopy equivalent to attaching a 1-simplex, $(x_1', x_2)$, and joining $N_{G_1'}(x_1') \ast N_{G_2'}(x_2)$. Similarly, the identification of $x_1''$ with $y_2$ gives rise to the attachment of $(x_1'',y_2)$ and joining $N_{G_1'}(x_1'') \ast N_{G_2'}(y_2)$. 

If $\mathcal{N}(G_2')$ is 0-connected (i.e. $\chi(G_2') \geq 3$) it suffices to show one of the connected components of $\mathcal{N}(G_1')$ contains at least two elements from $\{\{x_1'\}, \{x_1''\}, N_{G_1'}(x_1'), N_{G_1'}(x_1'')\}$ which follows from Lemma~\ref{lem:split} when $\chi(G_1') \geq 2$. So it suffices to consider $\chi(G_1') = \chi(G_2') = 2$. 

Suppose $\mathcal{N}(G_1'):= A \uplus B$ and $\mathcal{N}(G_2'):=C \uplus D$ have two connected components, respectively. From Lemma~\ref{lem:split}, we can assume $N_{G_1'}(x_1'), x_1'' \in A$ and $N_{G_1'}(x_1''), x_1' \in B$. From Lemma~\ref{lem:bipartite}, we can assume $x_2, y_2 \in C$ and $N_{G_2'}(x_2), N_{G_2'}(y_2) \in D$. Hence, up to homotopy equivalence the DGHO-composition attaches 1-cells between A and C, C and B, B and D, and D and A. The result follows. 
\end{proof}

%%%%%%%%%%%%%%%%%%%%%%%%%%%%%%%%%%%%%%%%%%%%%%%%%%%%%%%%%%%%

\section{Graph Construction Algorithms, Experimental Results, and Open Problems}\label{sec:experiments}

In this section, we report on results regarding computational experiments using SageMath~\cite{sagemath} via CoCalc.com~\cite{cocalc}.
We describe two stochastic algorithms that produce graphs using Haj\'os constructions and Urquhart constructions, respectively.
Using these algorithms, we generate sets of graphs and analyze the resulting distributions of their sizes, orders, and the ranks of the first homology groups of their neighborhood complexes.
We conclude the section with several open problems.

\subsection{Expected behavior of Haj\'os-type constructions and first Betti numbers}
Our motivation for analyzing the rank of the first homology group of $\N(G)$, called the first Betti number of $\N(G)$, is three-fold.
First, our main results in this work have shown that a graph $G$ obtained by a Haj\'os merge or a vertex identification with vertices at distance 5 or greater apart has at least one copy of $S^1$ as a wedge summand in $\N(G)$, up to homotopy.
This implies that the first Betti number of $\N(G)$ is strictly positive.
Second, if $\N(G)$ is highly connected, then a classical theorem in algebraic topology due to Hurewicz implies that the first Betti number is zero.
Thus, having a trivial first Betti number serves as an extremely weak approximation to being highly connected.
Third, Kahle has previously studied topological properties of neighborhood complexes of random graphs, establishing the following theorem.
Let the random graph $G(n,p)$ denote the probability space of all graphs on a vertex set of size $n$ with each edge inserted independently with probability $p$.
\begin{theorem}[Kahle~\cite{KahleRandomNeighborhood}]\label{thm:kahle}
  If $p=1/2$ and $\epsilon>0$ then almost always $\widetilde{H}_\ell(\N(G(n,p)))=0$ for $\ell\leq (1-\epsilon)\log_2(n)$.
\end{theorem}
From this perspective, if we were to sample uniformly from $n$-vertex graphs for large $n$, we would expect to see mainly graphs with trivial first Betti numbers.
Thus, we expect that a generic $k$-constructible graph on many vertices will have Haj\'os constructions that end in a vertex identification using two vertices that are at distance less than or equal to four from each other.

\subsection{Two Haj\'os-type construction algorithms}

We define in Appendix~\ref{sec:appendix} two algorithms that we call the Constructible Random Algorithm (CRA) and the Urquhart Random Algorithm (URA).
The CRA is a stochastic algorithm implementing the recursive definition of $k$-constructible graphs given in Defintion~\ref{def:hajosconst}. 
We use a probability $p$ to determine the likelihood of selecting a Haj\'os merge or vertex identification at each step of the algorithm; when $p$ is small, vertex identifications are favored.
The URA is based on repeatedly generating graphs using a slightly-restricted version of the Urquhart construction in Definition~\ref{def:urquhartconst}, where the restriction ensures that the input graphs for the Urquhart construction are connected.
We use random choices for certain parameters and vertex/edge selections in the URA, leading to the stochasticity of the algorithm.

\subsection{Experimental Results}

Using an implementation in SageMath, we ran the Constructible Random Algorithm for $p\in \{0.02,0.1,0.5\}$ and $k\in \{3,4,5,6\}$, generating various numbers of graphs.
For each of the three values of $p$ when $k = 5$ and for $p=0.02$ when $k=6$, we include in Appendix~\ref{sec:appendix} a plot of the number of vertices versus number of edges for the graphs generated, and a histogram of the first Betti numbers of the resulting neighborhood complexes.
We also provide in Table~\ref{tab:cradata} the percent of graphs for which the first Betti number was trivial.
Similarly, we ran the Urquhart Random Algorithm for $k\in \{3,4,5,6\}$ and $m=n=12$; we provide in Appendix~\ref{sec:appendix} and Table~\ref{tab:uradata} the same data report as for our CRA generated data.

\begin{table}
\begin{tabular}{|r|r|r|r|r|}
\hline
$k$ & $p$ & \text{\# graphs generated} & \text{fraction of first Betti numbers equal to zero}  \\
\hline
\hline
$3$ & $0.02$ & 10,000 & 0.25 \\
\hline
$3$ & $0.10$ & 10,000 & 0.065 \\
\hline
$3$ & $0.50$ & 10,000 & 0.0018 \\
  \hline
  \hline
$4$ & $0.02$ & 10,000 & 0.61 \\
\hline
$4$ & $0.10$ & 10,000 & 0.32 \\
\hline
$4$ & $0.50$ & 10,000 & 0.01\\
  \hline
  \hline
$5$ & $0.02$ & 10,000 & 0.88 \\
\hline
$5$ & $0.10$ & 10,000 & 0.54\\
\hline
$5$ & $0.50$ & 2,939 & 0.03\\
  \hline
  \hline
$6$ & $0.02$ & 10,000 &0.88 \\
\hline
$6$ & $0.10$ & 10,000 & 0.44\\
\hline
$6$ & $0.50$ & 2,894 & 0.08 \\
\hline
\end{tabular}
\caption{CRA results.}
\label{tab:cradata}
\end{table}

\begin{table}
\begin{tabular}{|r|r|r|}
\hline
$k$ & \text{\# graphs generated} & \text{fraction of first Betti numbers equal to zero} \\
\hline
  \hline
  $3$ & $2000$ & 0.67\\
  \hline
  $4$ & $2000$ & 0.74\\
  \hline
  $5$ & $2000$ & 0.81 \\
  \hline
  $6$ & $2000$ & 0.83\\
  \hline
\end{tabular}
\caption{URA results, with $m=n=12$ for all cases.}
\label{tab:uradata}
\end{table}

It is interesting to note that the URA-generated graphs frequently have zero first Betti number as shown in Figures~\ref{fig:urd4_betti},~\ref{fig:urd5_betti}, and~\ref{fig:urd6_betti} and Table~\ref{tab:uradata}, and also have a reasonable distribution of number of vertices versus number of edges, as shown in Figures~\ref{fig:urd4_os},~\ref{fig:urd5_os}, and~\ref{fig:urd6_os}.
On the one hand, this matches our expectation from Kahle's results that many graphs will have a trivial first Betti number; on the other hand, it is somewhat surprising that we so frequently have an Urquhart construction with final step being a vertex identification of two vertices at short distance from each other.
We obtain these distributions with only $2000$ graphs sampled.
It is worth noting that by the nature of Urquhart constructions beginning from supergraphs of $K_k$, the URA produces graphs that would potentially require a large number of recursive steps in order to be produced by the standard Haj\'os constructions.

The CRA data is somewhat more complicated, in that the probability $p$ that is selected as input for the algorithm plays a significant role in the outcomes.
Further, in order to obtain interesting data, it is necessary to generate a larger number of graphs; while it is possible for all of the cases we consider to generate $10000$ graphs with CRA, it was not always reasonable using SageMath via CoCalc to compute the first Betti numbers for all these graphs, which is why a smaller sample size is included in Table~\ref{tab:cradata} for $k=5$ and $k=6$ with $p=0.50$.
When $p=0.50$, Haj\'os merges are equally likely as vertex identifications in CRA, so it is not surprising that Table~\ref{tab:cradata} and Figure~\ref{fig:crd5_50_betti} show that most of the graphs produced have a positive first Betti number.
Also, because a Haj\'os merge of $G_1$ and $G_2$ results in a graph with $|V(G_1)|+|V(G_2)|-1$ vertices and $|E(G_1)|+|E(G_2)|-1$ edges, the plot in Figure~\ref{fig:crd5_50_os} is not particularly surprising.
While Table~\ref{tab:cradata} and Figures~\ref{fig:crd5_10_os} and~\ref{fig:crd5_10_betti} show that the situation improves for $p=0.10$, it is when $p=0.02$ that we see behavior similar to the URA output.

In summary, based on these initial investigations, the URA appears to be quite effective at producing sample sets of graphs that have a broad distribution of number of vertices against number of edges and also have a large percentange of graphs with zero first Betti number.
The CRA is also capable of generating reasonable sample sets, but it is less clear how long it would take the CRA with $p\approx 0.02$ to produce a sample set of graphs with a large average number of vertices.
To accomplish this task, the URA appears to be better suited.

\subsection{Further questions}

The original motivation inspiring the definition of $\N(G)$ was to provide lower bounds for the chromatic numbers of Kneser graphs~\cite{LovaszChromaticNumberHomotopy}.
These topological approaches have also been used subsequently to find sharp lower bounds for chromatic numbers of other graphs, e.g. the stable Kneser graphs~\cite{SchrijverGraphs}.
In this work we have shown that Haj\'os-type constructions of these and other graphs with highly-connected neighborhood complexes are constrained in significant ways, leading to the following:
\begin{problem}
  Find Haj\'os, Ore, and/or Urquhart constructions for Kneser and stable Kneser graphs, or for other graphs with highly-connected neighborhood complexes.
\end{problem}

Both Urquhart's results and our experimental data has demonstrated that, while the operations of Haj\'os merge and vertex identification leading to Haj\'os constructibility are foundational in all this work, Urquhart constructions of graphs as given in Definition~\ref{def:urquhartconst} serve as both a powerful theoretical tool and a useful ingredient in algorithms for sampling $k$-constructible graphs.
This motivates their further study, and thus we define the following families of graphs.
\begin{definition}
  Let $k\geq 3$, $m\geq k$, and $n\geq 1$, and define $\urq(k,m,n)$ to be the set of graphs that can be obtained as $\urq(G_1,\ldots,G_j)$ where
  \begin{itemize}
  \item each $G_i$ is connected,
  \item each $G_i$ is a supergraph of $K_k$,
  \item $|V(G_i)|\leq m$ for each $i$, and
  \item $1\leq j\leq n$.
  \end{itemize}
\end{definition}

Thus, $\urq(k,m,n)$ is the set of graphs that are Urquhart-$k$-constructible using at most $n$ component graphs (all connected), each having no more than $m$ vertices.
Given a $k$-constructible graph $G$ with $k\geq 3$, Urquhart's proof of Theorem~\ref{thm:urquhart} implies that there exist $m,n$ such that $G\in \urq(k,m,n)$.
The following problems are inspired by both classical questions in $k$-constructibility and topological properties of neighborhood complexes.

\begin{problem} 
  For a fixed $k\geq 3$, find families of graphs $G$ for which it is possible to find bounds on the values of $m$ and $n$ such that $G\in \urq(k,m,n)$.
\end{problem}

\begin{problem}
  What is the distribution of the chromatic numbers of the graphs in $\urq(k,m,n)$?
\end{problem}

\begin{problem}
  For each fixed $i$, what is the distribution of the ranks of the $i$-th homology groups for $\N(G)$ over $G\in\urq(k,m,n)$?
\end{problem}

%%%%%%%%%%%%%%%%%%%%%%%%%%%%%%%%%%%%%%%%%%%%%%%%%%%%%%%%%%%% 

\bibliographystyle{plain}
\bibliography{Braun}

% this places the addresses after the bibliography but prior to the appendices
\addresseshere

\appendix
\newpage

\section{Algorithms and Experimental Data}\label{sec:appendix}

\IncMargin{1em}
\begin{algorithm}
  \SetKwInOut{Input}{Input}\SetKwInOut{Output}{Output}
  \Input{A complete graph on $k$ vertices, a probability $p \in (0,1)$, and a positive integer $t$.}
  \Output{A list \texttt{GraphList} of $t$ $k$-constructible graphs.}
  \BlankLine
  Set $i = 0$\;
  Initialize \texttt{GraphList} as a list containing only $K_k$\;
  \While{$i<t$}{
    Generate a random value $0 < r <1$\;
    \eIf{\emph{$r > p$ and $i>0$}}{
      Randomly select a non-complete graph $G$ from \texttt{GraphList}\;
      Randomly choose a pair of non-adjacent vertices $(v,w)$ in $G$\;
      Set $G':=\vid(G,[v,w])$\;
    }
    {Randomly select two graphs (possibly equal) $G_1$ and $G_2$ from \texttt{GraphList}\;
      Randomly select one edge from each of $G_1$ and $G_2$\;
      Set $G':=G_1\haj G_2$, performing the Haj\'os merge with the two selected edges\;
    }
    \eIf{\emph{$G'$ is not isomorphic to any element of \texttt{GraphList}}}{
      Append $G'$ to \texttt{GraphList}\;
      Increase $i$ by one\;
    }{Continue\;}
  }
  \caption{Constructible Random Algorithm (CRA)}\label{alg:cra}
\end{algorithm}
\DecMargin{1em}

\IncMargin{1em}
\begin{algorithm}
  \SetKwInOut{Input}{Input}\SetKwInOut{Output}{Output}
  \Input{Positive integers $k\geq 3$ and $t,m,n\geq 1$.}
  \Output{A list \texttt{GraphList} of $t$ Urquhart-$k$-constructible graphs. }
  \BlankLine
  Set $i = 0$\;
  Initialize \texttt{GraphList} as an empty list\;
  \While{$i \leq t$}{
    Generate a random integer $1 \leq  r_{c} \leq n$\;
    Initialize $L$ as an empty list\;
    \For{$1\leq s \leq r_c$}{
      Set $G_s=K_k$\; 
      Generate a random integer $1 \leq r \leq m$\;
      Add $r$ new vertices named $v_1,\ldots,v_{r}$ to $G_s$\;
      \For{$1\leq j \leq r$}{Select a random subset $\emptyset \neq S \subseteq V(K_k)$\;
        Add to $G_s$ the edges $\{(v_j,w):w\in S\}$\;
      }
      Select a random graph $H$ on vertex set $\{v_1,\ldots,v_r\}$\;
      Add the edges of $H$ to $G_s$\;
      Append $G_s$ to $L$\;
    }
    \While{\emph{$L$ contains more than one element}}{
      Randomly select two graphs $G_1, G_2 \in L$\;
      Delete $G_1$ and $G_2$ from $L$\;
      Randomly select one edge from each of $G_1$ and $G_2$\;
      Choose a random integer $1\leq \ell\leq \min\{|V(G_1)|, |V(G_2)|\} -1$\;
      Select $\ell$ pairs of disjoint vertices from $G_1$ and $G_2$ satisfying the criteria to be used in an Ore merge\;
      Append to $L$ the new graph $\ore(G_1, G_2)$ formed using these selections\; 
    }
    Define $G'$ to be the unique remaining element of $L$\;
    \eIf{ \emph{$G'$ is not isomorphic to any element of \texttt{GraphList}}}{
      Append $G'$ to \texttt{GraphList}\;
      Increase $i$ by one\;
    }{Continue\;}
  }
  \caption{Urquhart Random Algorithm (URA)}\label{alg:ura}
\end{algorithm}
\DecMargin{1em}

\newpage

\begin{figure}[h]
\centering
\begin{minipage}{.5\textwidth}
  \centering
  \includegraphics[width=\linewidth]{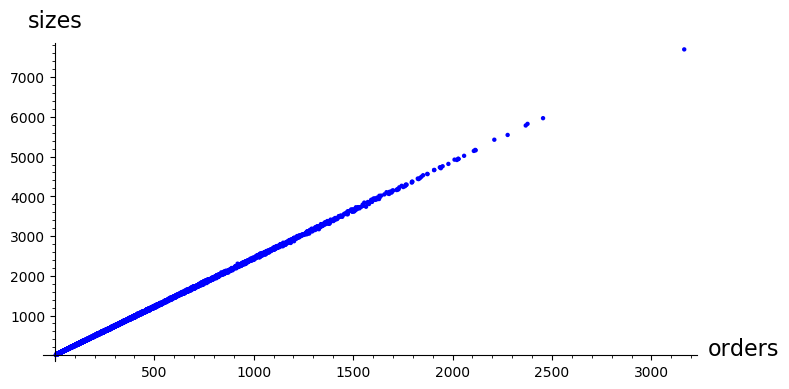}
  \caption{Orders versus sizes for $2939$ CRA-generated graphs, $k=5$, $p=0.50$}
  \label{fig:crd5_50_os}
\end{minipage}%
\begin{minipage}{.5\textwidth}
  \centering
  \includegraphics[width=\linewidth]{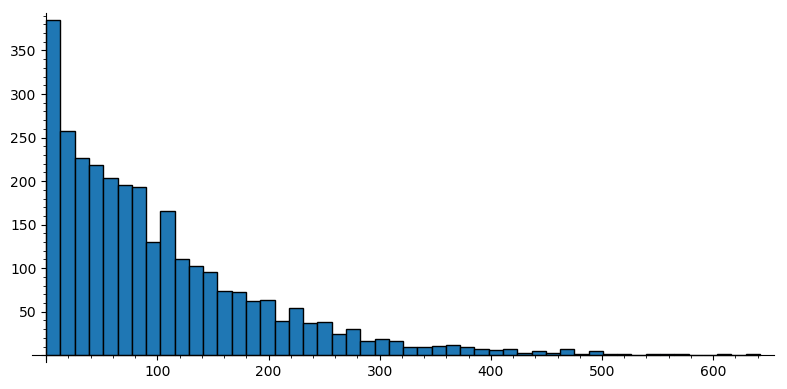}
  \caption{Histogram of first Betti numbers for $\N(G)$ for $2939$ CRA-generated graphs, $k=5$, $p=0.50$}
  \label{fig:crd5_50_betti}
\end{minipage}
\end{figure}

\begin{figure}[h]
\centering
\begin{minipage}{.5\textwidth}
  \centering
  \includegraphics[width=\linewidth]{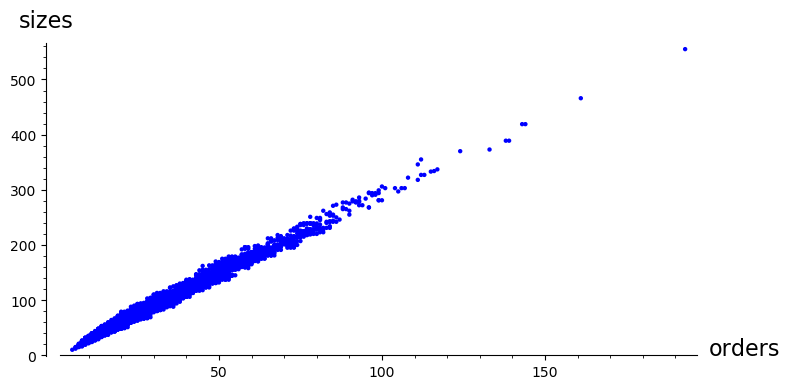}
  \caption{Orders versus sizes for $10000$ CRA-generated graphs, $k=5$, $p=0.10$}
  \label{fig:crd5_10_os}
\end{minipage}%
\begin{minipage}{.5\textwidth}
  \centering
  \includegraphics[width=\linewidth]{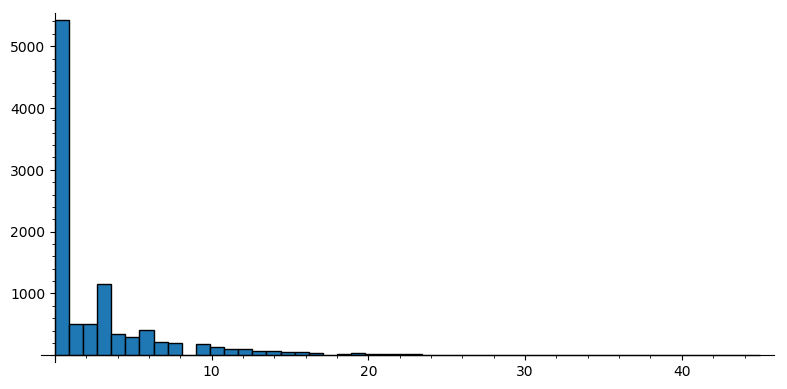}
  \caption{Histogram of first Betti numbers for $\N(G)$ for $10000$ CRA-generated graphs, $k=5$, $p=0.10$}
  \label{fig:crd5_10_betti}
\end{minipage}
\end{figure}

\begin{figure}[h]
\centering
\begin{minipage}{.5\textwidth}
  \centering
  \includegraphics[width=\linewidth]{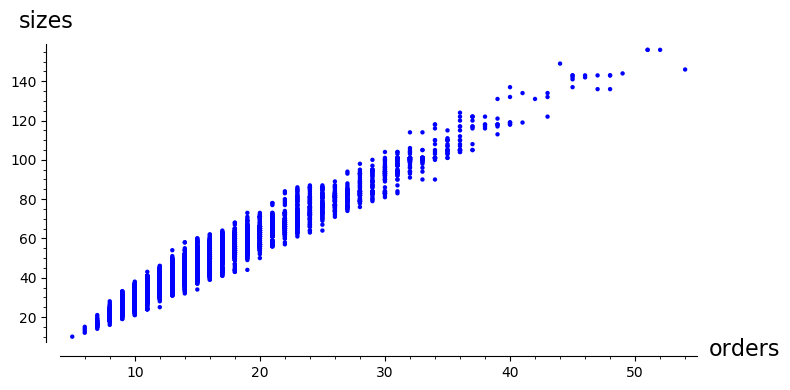}
  \caption{Orders versus sizes for $10000$ CRA-generated graphs, $k=5$, $p=0.02$}
  \label{fig:crd5_2_os}
\end{minipage}%
\begin{minipage}{.5\textwidth}
  \centering
  \includegraphics[width=\linewidth]{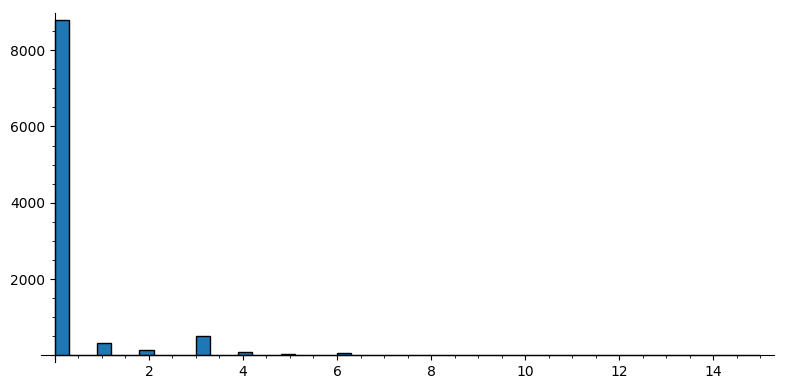}
  \caption{Histogram of first Betti numbers for $\N(G)$ for $10000$ CRA-generated graphs, $k=5$, $p=0.02$}
  \label{fig:crd5_2_betti}
\end{minipage}
\end{figure}

\begin{figure}[h]
\centering
\begin{minipage}{.5\textwidth}
  \centering
  \includegraphics[width=\linewidth]{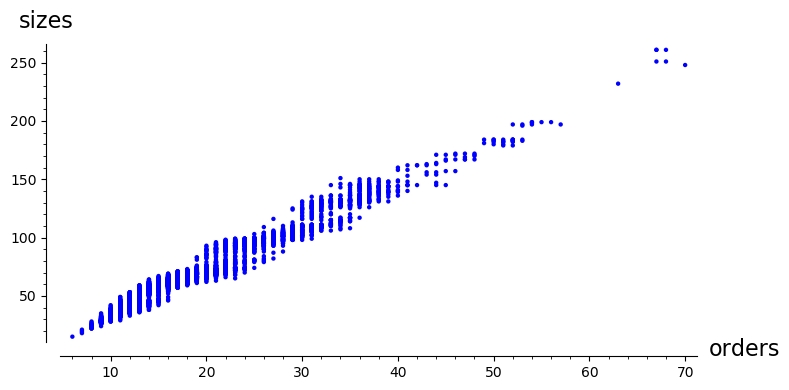}
  \caption{Orders versus sizes for $10000$ CRA-generated graphs, $k=6$, $p=0.02$}
  \label{fig:crd6_2_os}
\end{minipage}%
\begin{minipage}{.5\textwidth}
  \centering
  \includegraphics[width=\linewidth]{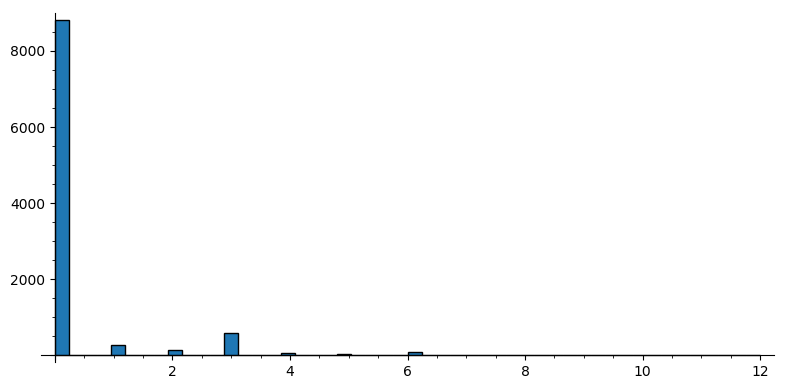}
  \caption{Histogram of first Betti numbers for $\N(G)$ for $10000$ CRA-generated graphs, $k=6$, $p=0.02$}
  \label{fig:crd6_2_betti}
\end{minipage}
\end{figure}

\begin{figure}[h]
\centering
\begin{minipage}{.5\textwidth}
  \centering
  \includegraphics[width=\linewidth]{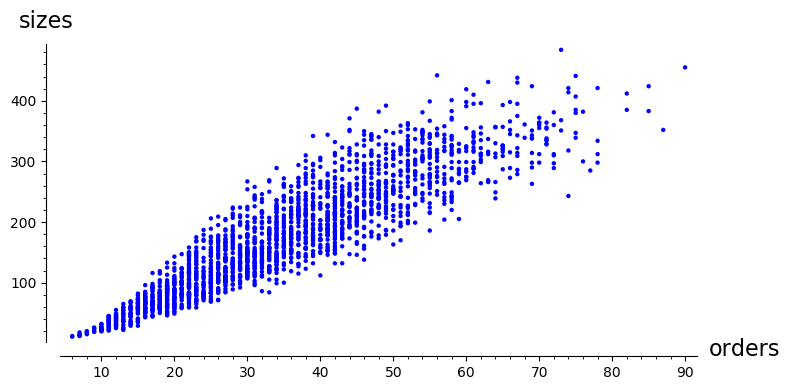}
  \caption{Orders versus sizes for $2000$ URA-generated graphs, $k=4$, $m=n=12$}
  \label{fig:urd4_os}
\end{minipage}%
\begin{minipage}{.5\textwidth}
  \centering
  \includegraphics[width=\linewidth]{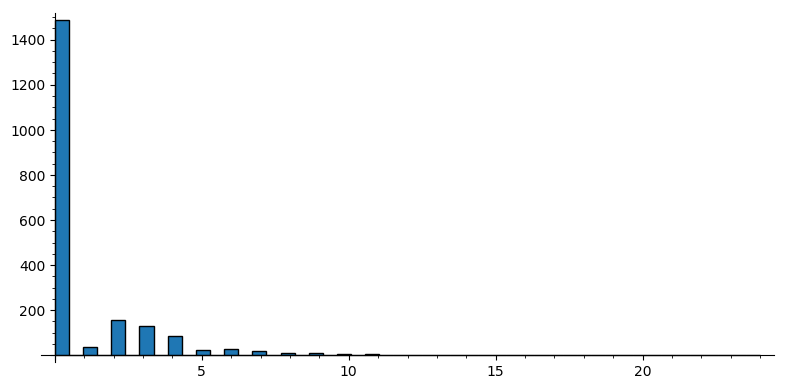}
  \caption{Histogram of first Betti numbers for $\N(G)$ for $2000$ URA-generated graphs, $k=4$, $m=n=12$}
  \label{fig:urd4_betti}
\end{minipage}
\end{figure}

\begin{figure}[h]
\centering
\begin{minipage}{.5\textwidth}
  \centering
  \includegraphics[width=\linewidth]{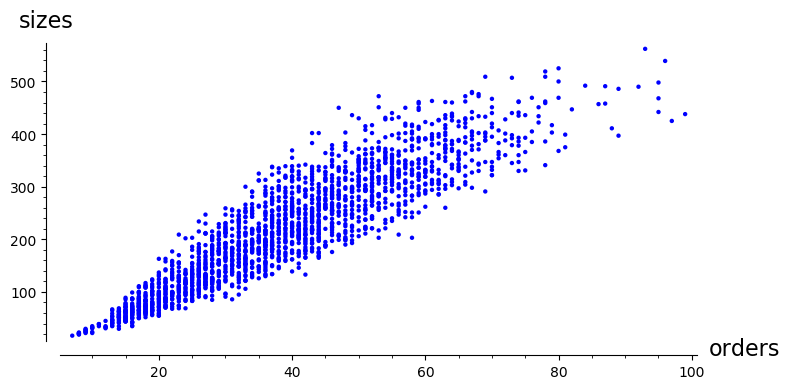}
  \caption{Orders versus sizes for $2000$ URA-generated graphs, $k=5$, $m=n=12$}
  \label{fig:urd5_os}
\end{minipage}%
\begin{minipage}{.5\textwidth}
  \centering
  \includegraphics[width=\linewidth]{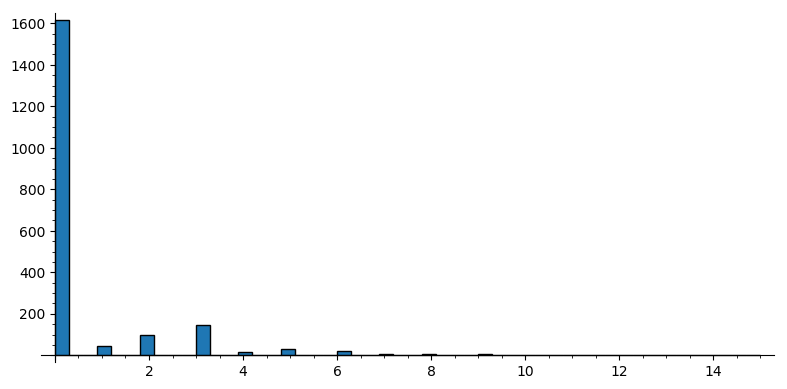}
  \caption{Histogram of first Betti numbers for $\N(G)$ for $2000$ URA-generated graphs, $k=5$, $m=n=12$}
  \label{fig:urd5_betti}
\end{minipage}
\end{figure}

\begin{figure}[h]
\centering
\begin{minipage}{.5\textwidth}
  \centering
  \includegraphics[width=\linewidth]{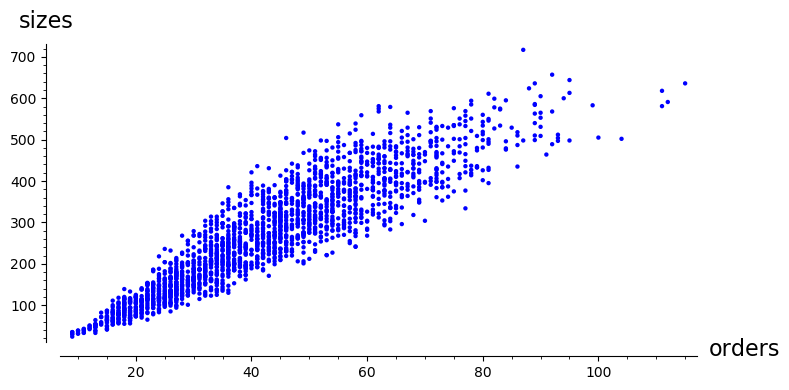}
  \caption{Orders versus sizes for $2000$ URA-generated graphs, $k=6$, $m=n=12$}
  \label{fig:urd6_os}
\end{minipage}%
\begin{minipage}{.5\textwidth}
  \centering
  \includegraphics[width=\linewidth]{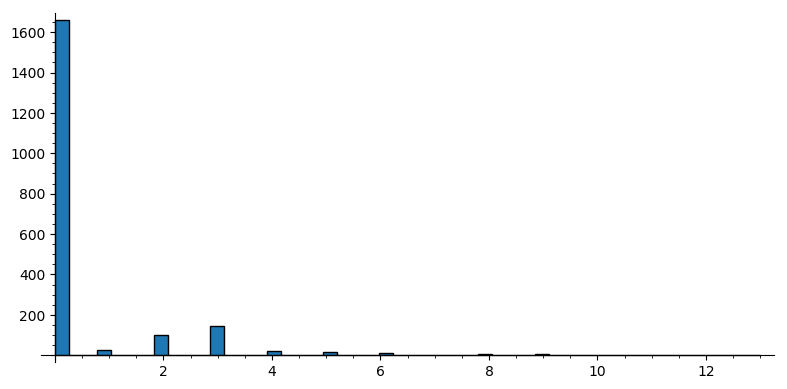}
  \caption{Histogram of first Betti numbers for $\N(G)$ for $2000$ URA-generated graphs, $k=6$, $m=n=12$}
  \label{fig:urd6_betti}
\end{minipage}
\end{figure}

\end{document}